\documentclass[onefignum,onetabnum]{siamart171218}

\definecolor{violet}{rgb}{0.580,0,0.827}

\newcommand{\myvec}[1]{\ensuremath{\begin{bmatrix}#1\end{bmatrix}}}

\usepackage[textsize=small]{todonotes}

\usepackage{lipsum}
\usepackage{amsfonts}
\usepackage{graphicx}
\usepackage{epstopdf}
\usepackage[section]{placeins}
\usepackage{multirow}
\usepackage{algorithm,algorithmic}
\usepackage{tikz}
\usepackage{caption}
\usetikzlibrary{shapes,arrows,positioning,shadows}
\usepackage{amsmath,bm,times}
\ifpdf
  \DeclareGraphicsExtensions{.eps,.pdf,.png,.jpg}
\else
  \DeclareGraphicsExtensions{.eps}
\fi

\newsiamremark{remark}{Remark}
\newsiamremark{hypothesis}{Hypothesis}
\crefname{hypothesis}{Hypothesis}{Hypotheses}
\newsiamthm{claim}{Claim}

\title{Data Assimilation with Deep Neural Nets Informed by Nudging
\author{Harbir Antil\thanks{Department of Mathematical Sciences and The Center for Mathematics and Artificial Intelligence, George Mason University, Fairfax, VA 22030 
  (\email{hantil@gmu.edu}).}
\and Rainald L\"ohner\thanks{The Center for Computational Fluid Dynamics, George Mason University, Fairfax, VA 22030 
  (\email{rlohner@gmu.edu}).}
\and Randy Price\thanks{The Center for Mathematics and Artificial Intelligence and The Center for Computational Fluid Dynamics, George Mason University, Fairfax, VA 22030 
  (\email{rprice25@gmu.edu}).}
}
\thanks{This work is partially supported by the Defense Threat Reduction Agency (DTRA) under contract HDTRA1-15-1-0068. Jacqueline Bell served as the technical monitor. Also, it is partially supported by 
NSF grants DMS-2110263, DMS-1913004, the Air Force Office of Scientific Research under Award NO: FA9550-19-1-0036.}}

\usepackage{amsopn}

\makeatletter
\newcommand*{\addFileDependency}[1]{
  \typeout{(#1)}
  \@addtofilelist{#1}
  \IfFileExists{#1}{}{\typeout{No file #1.}}
}
\makeatother

\ifpdf
\hypersetup{
  pdftitle={Data Assimilation with Deep Neural Nets Informed by Nudging},
  pdfauthor={Harbir Antil, Rainald Lohner, and Randy Price}
}
\fi

\begin{document}

\maketitle

\begin{abstract}
The nudging data assimilation algorithm is a powerful tool used to forecast phenomena of interest given incomplete and noisy observations. Machine learning is becoming increasingly popular in data assimilation given its ease of computation and forecasting ability. This work proposes a new approach to data assimilation via machine learning where Deep Neural Networks (DNNs) are being taught the nudging algorithm. The accuracy of the proposed DNN based algorithm is comparable to the nudging algorithm and it is confirmed by the Lorenz 63 and Lorenz 96 numerical examples. The key advantage of the proposed approach is the fact that, once trained, DNNs are cheap to evaluate in comparison to nudging where typically differential equations are needed to be solved. Standard exponential type approximation results are established for the Lorenz 63 model for both the continuous and discrete in time models. These results can be directly coupled with estimates for DNNs (whenever available), to derive the overall  approximation error estimates of the proposed algorithm.
\end{abstract}



\section{Introduction}
Data assimilation techniques are used to improve our knowledge about the state by combining the underlying model with potentially noisy and partial observations. Classically, this is done with the Kalman filter and its modifications, see for instance \cite{bocquet,kellylawstuart,Law2015DataAA,majda_harlim_2012, reich_cotter_2015}. An alternative approach is the nudging data assimilation algorithm which uses the model to nudge the solution towards the observations. Originally nudging was done in the context of  finite-dimensional dynamical systems governed by ordinary differential equations and meteorology \cite{DataAssimilationandInitializationofHurricanePredictionModels,npg-15-305-2008,TheInitializationofNumericalModelsbyaDynamicInitializationTechnique,364173,article,UseofFourDimensionalDataAssimilationinaLimitedAreaMesoscaleModelPartIExperimentswithSynopticScaleData}. Later nudging was generalized to other models and more general situations, including partial differential equations \cite{Albanez2016ContinuousDA,Bessaih2015ContinuousDA,biswas,Farhat2015ContinuousDA,Farhat2015DataAA,Farhat2016OnTC,Farhat2017ADA,foias2016discrete,Markowich2015ContinuousDA}. 

The key challenge is that the initial conditions $u_0$ in a continuous dynamical system are unknown
\begin{equation}\label{eq:orig}
\frac{d}{dt} u = F(u),\quad u(0)=u_0 .
\end{equation}
Therefore, one cannot use the dynamical system \eqref{eq:orig} to make future predictions. 
Assuming that some partial observations of $u$ are available, the nudging algorithm entails solving an associated system

\begin{equation} \label{aotsystem}
\frac{d}{dt} w =F(w)-\mu (I_M w-I_M u),\quad w(0)=w_0\ (\mbox{arbitrary}),
\end{equation}
where $I_M$ is a linear operator called the interpolant operator. Here $M$ refers to the number of observations  and $\mu>0$ is the nudging parameter. In various cases, it is possible to establish approximation error estimates between $u$ and $w$ solving \eqref{eq:orig} and \eqref{aotsystem}, respectively. For instance, in case of Navier-Stokes equations, recently exponential estimates has been derived in \cite{biswas} for continuous in time observations and  in \cite{foias2016discrete} for discrete in time observations.

Data assimilation techniques have proven their effectiveness and now there are efforts to incorporate machine learning into the algorithms to either improve forecasting ability or reduce computation time. In \cite{app11031114,pawar2020long} the authors learn the misfit between a data assimilation algorithm and the dynamical model in order to improve forecasting ability. In \cite{brajard2021combining} the authors combine data assimilation with a neural network trained to learn the model error which arises when we don't have access to (\ref{eq:orig}) but only an approximation. In \cite{Brajard2021} the authors train a surrogate model with data from the ensemble Kalman filter to replace the dynamical system and are able to recover some of the forecasting ability with noisy observations. 

This work presents a new approach to data assimilation via machine learning. In particular, we combine the two fields to reduce computation time by learning the nudged equations (\ref{aotsystem}). We introduce a new algorithm which uses the training data from nudging, notice that this training data can be generated using any other approach, such as ensemble Kalman filter, as well. Thus the deep neural networks (DNNs) are trained to effectively learn the data assimilation algorithm. We show that the DNN based algorithm retains the accuracy of the nudging algorithm and is cheap to evaluate. The effectiveness of the proposed algorithm is shown on the Lorenz 63 and Lorenz 96 models in which case (\ref{aotsystem}) is a system of ODEs.  

DNNs are known to learn ODEs/PDEs \cite{HAntil_HCElman_AOnwunta_DVerma_2021a,brown2021novel,Ghosh2020STEERS,ji2020stiff,PINN}. In case of stiff or chaotic systems \cite{brown2021novel}, one idea is to learn one time step map, i.e., learn the update map $u^{n-1}\to u^n$. The neural net takes as input the state at time $t_{n-1}$ and the output is the approximate state at time $t_n$. This is the approach we will take, the details are given in section \ref{sec:method}. We use the newly introduced residual neural networks (ResNets), with bias ordering \cite{antil2021deep}, to carry out the training, but other types of DNN strategies are also equally applicable. ResNets are a popular way to carry out supervised learning, their connection to ODEs has been exploited to establish their stability \cite{KHe_XZhang_SRen_JSun_2016a,EHaber_LRuthotto_2018a,LRuthotto_EHaber_2019a}. In addition, the bias order reduces the parameter search space, it also provides more accurate results than standard DNNs, see \cite{antil2021deep} for a numerical evidence.

In the case of the Lorenz 63 model we prove exponential error estimates between the nudging solution and the reference solution for both the time continuous and fully discrete cases. These error estimates can be directly combined with DNN error estimates (whenever available).

\medskip
\noindent
{\bf Outline:} Section~\ref{sec:obj} provides our problem statement, which is followed by section~\ref{sec:method} where we introduce our algorithm. Section~\ref{sec:error}, focuses on error analysis. As an example, we first provide exponential type approximation error estimates in case of the nudging applied to the Lorenz 63 model. The is followed by a discussion of the overall error estimates due to the DNN approximation. Next, we provide an experimental protocol in section~\ref{sec:protocol}, which is followed by our experimental results on Lorenz 63 and 96 models in section~\ref{sec:experiments}.

\section{Objectives and Definitions}\label{sec:obj}

The goal of this section is to restate the nudging algorithm and recall the DNN algorithm with bias ordering. This section will set the stage for our proposed algorithm in next section.

\subsection{Nudging}
We consider a chaotic dynamical system represented by a differential equation
\begin{equation}\label{ODE}
    \frac{du}{dt} = f(u(t)),
\end{equation}
with $u(t)\in X$ which we take as the ``truth" or reference solution. Here $X$ is either an infinite dimensional space (in case of PDEs) or a finite dimensional space $X = \mathbb{R}^d$ in case of vector ODEs. The challenge is that the initial condition $u(0)$ is an unknown, therefore we cannot identify the dynamics of $u$.

Next, we introduce a continuous nudging strategy
\begin{equation} \label{nudging_cont}
\frac{dw}{dt} = f(w)-\mu (I_M w-I_M u),\quad w(0)=w_0\ (\mbox{arbitrary}),
\end{equation}
where $I_M : X \rightarrow Y$ is known as an interpolant operator. Here $Y$ is an infinite dimensional function space in the case of PDEs and $Y = \mathbb{R}^{d^*}$ ($d^*\le d)$ for ODEs. Moreover, $M$ refers to the number of observations and $\mu$ is the nudging parameter. One example of $I_M$ is the orthogonal projection operator $P_K$. Here, $P_K$ denotes the orthogonal projection onto a set $K$.
In the case $X=\mathbb{R}^d$, $K = \mbox{span} \{ \phi_i \}_{i=1}^K$ with  $\phi_i = \textbf{e}_i$ the standard basis elements in $\mathbb{R}^d$.

Notice that, the dynamics of $w$ are known, therefore \eqref{nudging_cont} does not suffer from the same issues as \eqref{ODE}. The goal of nudging is to select $\mu > 0$ to ensure that the nudging solution $w$ is tracking the observations $\{I_M u\}$. Thus we can replace \eqref{ODE} by \eqref{nudging_cont}.

From a practical point of view, it is essential to consider the corresponding discrete nudging data assimilation algorithm. This can be represented by the altered system 
\begin{align}\label{nudging}
\begin{aligned}
    \frac{dw}{dt} &= f(w(t)) - \mu(I_M(w(t_n))-I_M(u(t_{n}))),\quad \forall t\in[t_n,t_{n+1}], \\
    w(0)&=w_0,
\end{aligned}    
\end{align}
 $\{t_n\}_{n=1}^N$ are the discrete times at which we have observations of the reference solution. Note that in (\ref{nudging}), $\{I_M(u(t_n))\}_{n=1}^N$ is an input to the algorithm and represents the observations we have of the true solution to \eqref{ODE}.

\subsection{DNN with bias ordering} 
Next we consider the input-to-output map $w\to S(w)$ where $S$ is the map taking $w(t_{n-1})$ to $w(t_n)$ using (\ref{nudging}). The goal is to learn an approximation $\hat{S}$ of $S$ using Deep Neural Networks (DNNs). We consider the DNNs introduced in \cite{antil2021deep}, which in an abstract form amounts to solving an optimization problem 
\begin{align}\label{NN}
    \min_{\{W_\ell\}_{\ell=0}^{L-1},\{b_\ell\}_{\ell=0}^{L-2}} J(\{(y_L^i,S(u^i))\}_i,\{W_\ell\}_\ell,\{b_\ell\}_\ell), \\
    \text{subject to }y_L^i = \mathcal{F}(u^i;(\{W_\ell\},\{b_\ell\})),\quad i=1,...,N_s, \label{DNN} \\
    b_\ell^j\le b_\ell^{j+1},\quad j=1,...,n_{\ell+1}-1,\quad \ell=0,...,L-2, \label{biasordering} 
\end{align}
In \eqref{NN}, $J$ represents a loss function and we have $N_s$ samples for training. The equation \eqref{DNN} contains the DNN, in our numerical experiments we consider residual neural networks (ResNet). Moreover, 
the weight matrix $W_\ell\in\mathbb{R}^{n_\ell\times n_{\ell+1}}$, the bias vector $b_\ell\in\mathbb{R}^{n_\ell +1}$ where the $\ell$-th layer has $n_\ell$ neurons. Notice that,  \eqref{biasordering} enforces bias ordering in each layer, see \cite{antil2021deep} for more details. 

In our numerical experiments, we consider a quadratic loss  function, i.e., 
\begin{align}\label{lossf}
J:=\frac{1}{2N}\sum_{i=1}^N\|y_L^i-S(u^i)\|_2^2 + \frac{\lambda}{2}\sum_{\ell=0}^{L-1}(\|W_\ell\|_1 + \|b_\ell\|_1 + \|W_\ell\|_2^2 + \|b_\ell\|_2^2) ,
\end{align}
where the second term corresponds to regularization of weights and biases with $\lambda \ge 0$ denoting a regularization parameter.

Following \cite{antil2021deep}, we impose \eqref{biasordering} by using Moreau-Yosida penalty approach.
%
\begin{equation}
    J_{\gamma} := J + \frac{\gamma}{2}\sum_{\ell=0}^{L-2}\sum_{j=1}^{n_{\ell+1}-1}\|\min\{b_{\ell}^{j+1}-b_{\ell}^j,0\}\|_2^2,
\end{equation}
where $\gamma$ is the penalization parameter. For convergence results as $\gamma \rightarrow \infty$ we refer to \cite{antil2021deep}.

\subsection{Residual Neural Networks (ResNets)}
The results obtained below are using\\ ResNets with the following form
\begin{align}\label{ResNet1}
    \mathcal{F} &\ = f_{L-1}\circ f_{L-2} \circ ... \circ f_0, \\
    f_0(y_0) &:= \sigma(W_0y_0+b_0),\\
    f_l(y_\ell) &:= y_\ell + \tau\sigma(W_\ell y_\ell + b_\ell),\quad \ell=1,...,L-2,\\
    f_{L-1} &:= W_{L-1}y_{L-1},\label{ResNet4}
\end{align}
for scalar $\tau>0$ and the activation function is taken to be a smooth quadratic approximation of the ReLU function,
\begin{equation}\label{sigma}
 \sigma(x) =
    \begin{cases}
     \max\{0,x\} & |x|>\epsilon,\\
      \frac{1}{4\epsilon}x^2+\frac{1}{2}x+\frac{\epsilon}{4} & |x|\le\epsilon.\\
    \end{cases}     
\end{equation}
For more on these networks, see \cite{Antil2021NovelDN,antil2020fractional,paraDNN} among others.

The objective of this work is to learn the nudging equation  \eqref{nudging} by training the DNN in \eqref{NN}-\eqref{biasordering} with training data consisting of  
observations $\{I_M(u(t_n))\}_{n=1}^N$ and solution to \eqref{nudging}.
Similarly to nudging (cf.~\eqref{nudging}), in this way, the DNN (cf.~\eqref{DNN}) is able to learn how to track the true solution (cf.~\eqref{ODE}) from incomplete data.

\section{Method}\label{sec:method}

Let us assume that incomplete observations of the truth are available at observation times $\{t_k\}_{k=1}^{N}$. The training data for the ResNet,

\begin{equation}\label{trainingdata}
    \bigg\{\bigg\{\bigg(\myvec{w^i(t_k)\\I_M(u^i(t_k))},w^i(t_{k+1})\bigg)\bigg\}_{k=1}^{N}\bigg\}_{i=1}^{N_s},
\end{equation}
has $N \times N_s$ samples or input-output pairs. Each $i$ corresponds to a different reference solution which is generated by solving (\ref{ODE}) for different initial conditions. The pair \\ $((w^i(t_k),I_M(u^i(t_k)),w^i(t_{k+1}))$ is then generated by solving the corresponding nudging ODE (\ref{nudging}) on the interval $[t_k,t_{k+1}]$ with initial data $w^i(t_k)$ and given observations\\ $I_M(u^i(t_k))$. See Algorithm \ref{alg:offline} and Figure~\ref{diagram:offline} on how to prepare the input-output training data for the DNN.

\begin{algorithm}[h!]
\caption{Offline Data Preparation to Train DNN}\label{alg:offline}
\begin{algorithmic}[1]
\STATE Generate $N_s$ initial conditions $\{u^i(0)\}_{i=1}^{N_s}$ for (\ref{ODE}).
\STATE Choose final time $T$ based on number of training samples desired.
\STATE Solve \eqref{ODE} on $[0, T]$, for each initial condition, to generate $\{u^i(t)\}_{i=1}^{N_s}$ for $t\in[0,T]$.
\FOR{$i \in \{0,\dots,N_s\},\ k\in \{1,\dots,N\}$}
        \IF{$k=0$}
        \STATE Initialize $w^i(t_0)=0$ (or arbitrary).
        \ENDIF
        \STATE Solve (\ref{nudging}) on $[t_k,t_{k+1}]$ with observation $I_M(u^i(t_k))$ and initial condition $w^i(t_k)$.
        \STATE Store $\text{INPUT}(i,k) = [w^i(t_k),I_M(u^i(t_k))]$.
        \STATE Store $\text{OUTPUT}(i,k) = [w^i(t_{k+1})]$.
        \ENDFOR
\STATE Return: INPUT and OUTPUT to train the DNN.
\end{algorithmic}
\end{algorithm}

\pgfdeclarelayer{background}
\pgfdeclarelayer{foreground}
\pgfsetlayers{background,main,foreground}

\tikzstyle{sensor}=[draw, fill=blue!20, text width=4.1em, 
    text centered, minimum height=2.5em]
\tikzstyle{sensor1}=[draw, fill=blue!20, text width=2.4em, 
    text centered, minimum height=2.5em]
\tikzstyle{ann} = [above, text width=5em]
\tikzstyle{naveqs} = [sensor, text width=2em, fill=red!20, 
    minimum height=5em, rounded corners]
\tikzstyle{nudge} = [sensor, text width=2.8em, fill=red!20, 
minimum height=6em, rounded corners]
\def\blockdist{2.5}
\def\edgedist{2.5}
\begin{flushleft}
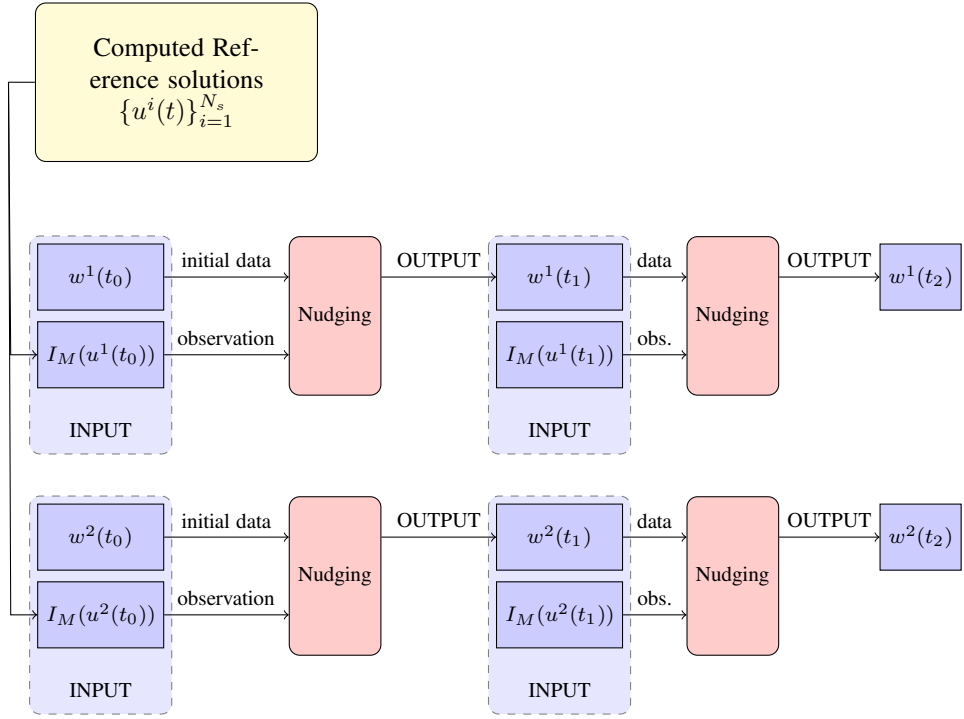
\begin{figure}
\begin{tikzpicture}
    \node (naveq) [nudge] {\footnotesize Nudging};
    \node (naveq1) [nudge,right of=naveq, xshift=12.2em] {\footnotesize Nudging};
    \node (naveq2) [sensor, text width=10em, fill=yellow!20, 
minimum height=6em, rounded corners, above of=naveq, yshift=6em, xshift = -6em]{Computed Reference solutions\\
$\{u^{i}(t)\}_{i=1}^{N_s}$};
    \path (naveq.140)+(-\blockdist,0) node (gyros) [sensor] {\footnotesize $w^1(t_0)$};
    \path (naveq.140)+(8.4,0) node (gyros2) [sensor1] {\footnotesize $w^1(t_2)$};
    \path (naveq.140)+(3.6,0) node (gyros1) [sensor] {\footnotesize $w^1(t_1)$};
    \path (naveq.-140)+(3.6,0) node (accel1) [sensor] {\footnotesize $I_M(u^1(t_1))$};
    \path (naveq.-140)+(-\blockdist,0) node (accel) [sensor] {\footnotesize $I_M(u^1(t_0))$};
    \path [draw, ->] (gyros) -- node [above] {\footnotesize initial data} 
        (naveq.west |- gyros) ;
    \path [draw, ->] (gyros1) -- node [above] {\footnotesize data} 
        (naveq1.west |- gyros1) ;
    \path [draw, ->] (accel) -- node [above] {\footnotesize observation} 
        (naveq.west |- accel);
    \path [draw, ->] (accel1) -- node [above] {\footnotesize obs.} 
    (naveq1.west |- accel1);
    \node (IMU) [below of=accel] {\footnotesize INPUT};
    \node (IMU1) [below of=accel1] {\footnotesize INPUT};
    \path (naveq.south west)+(-0.6,-0.4) node (INS) {};
    \path [draw, <-] (gyros1) -- node [above] {\footnotesize OUTPUT} 
    (naveq.east |- gyros1) ;
    \path [draw, <-] (gyros2) -- node [above] {\footnotesize OUTPUT} 
    (naveq1.east |- gyros2) ;
    
    \node (naveq20) [nudge,below of=naveq, yshift=-7em] {\footnotesize Nudging};
    \node (naveq21) [nudge,right of=naveq20, xshift=12.2em] {\footnotesize Nudging};
    \path (naveq20.140)+(-\blockdist,0) node (gyros20) [sensor] {\footnotesize $w^2(t_0)$};
    \path (naveq20.140)+(8.4,0) node (gyros22) [sensor1] {\footnotesize $w^2(t_2)$};
    \path (naveq20.140)+(3.6,0) node (gyros21) [sensor] {\footnotesize $w^2(t_1)$};
    \path (naveq20.-140)+(3.6,0) node (accel21) [sensor] {\footnotesize $I_M(u^2(t_1))$};
    \path (naveq20.-140)+(-\blockdist,0) node (accel20) [sensor] {\footnotesize $I_M(u^2(t_0))$};
    \path [draw, ->] (gyros20) -- node [above] {\footnotesize initial data} 
        (naveq20.west |- gyros20) ;
    \path [draw, ->] (gyros21) -- node [above] {\footnotesize data} 
        (naveq21.west |- gyros21) ;
    \path [draw, ->] (accel20) -- node [above] {\footnotesize observation} 
        (naveq20.west |- accel20);
    \path [draw, ->] (accel21) -- node [above] {\footnotesize obs.} 
    (naveq21.west |- accel21);
    \node (IMU20) [below of=accel20] {\footnotesize INPUT};
    \node (IMU21) [below of=accel21] {\footnotesize INPUT};
    \path (naveq20.south west)+(-0.6,-0.4) node (INS20) {};
    \path [draw, <-] (gyros21) -- node [above] {\footnotesize OUTPUT} 
    (naveq20.east |- gyros21) ;
    \path [draw, <-] (gyros22) -- node [above] {\footnotesize OUTPUT} 
    (naveq21.east |- gyros22) ;
    \coordinate[left = 1em of naveq2] (Empty1);
    \coordinate[left = 1em of accel] (Empty2);
    \coordinate[left = 1em of accel20] (Empty3);
    \path [draw, ->] (naveq2) -- (Empty1) --  (Empty2)  |- (accel) ;
    \path [draw, ->] (Empty1) --  (Empty3)  |- (accel20) ;

    \begin{pgfonlayer}{background}
        \path (gyros.north west)+(-0.1,0.1) node (a) {};
        \path (IMU.south -| gyros.east)+(+0.1,-0.1) node (b) {};
        \path[fill=blue!10,rounded corners, draw=black!50, dashed]
            (a) rectangle (b);
        \path (gyros1.north west)+(-0.1,0.1) node (a) {};
        \path (IMU1.south -| gyros1.east)+(+0.1,-0.1) node (b) {};
        \path[fill=blue!10,rounded corners, draw=black!50, dashed]
            (a) rectangle (b);
        \path (gyros20.north west)+(-0.1,0.1) node (a) {};
        \path (IMU20.south -| gyros20.east)+(+0.1,-0.1) node (b) {};
        \path[fill=blue!10,rounded corners, draw=black!50, dashed]
            (a) rectangle (b);
        \path (gyros21.north west)+(-0.1,0.1) node (a) {};
        \path (IMU21.south -| gyros21.east)+(+0.1,-0.1) node (b) {};
        \path[fill=blue!10,rounded corners, draw=black!50, dashed]
            (a) rectangle (b);
    \end{pgfonlayer}
\end{tikzpicture}
\caption{Generating the training data (\ref{trainingdata}) with arbitrary initial data $w^i(t_0)$ for $N_s = 2$ reference solutions. Each input-output pair pictured is then used to train the DNN.}
\label{diagram:offline}
\end{figure}
\end{flushleft}

\begin{remark}
Notice that, we are using nudging to generate the training data for the DNN. It is important to generate training data closer to the initial start time rather than generating data after nudging algorithm has already converged.
\end{remark}

Algorithm~\ref{alg:online} and Figure~\ref{diag:online} describes how to 
use the proposed DNN approach to carry out data assimilation. Approximation error estimates for this algorithm will be discussed in the next section.

\begin{algorithm}
\caption{Online Usage of DNN for Data Assimilation}\label{alg:online}
\begin{algorithmic}[1]
\STATE Initialize $w_{DNN}(t_0)=0$ (or arbitrary)
\STATE Observe $I_M(u^i(t_0))$
\STATE Set $k=0$
\WHILE{Observations are available}
\STATE $w_{DNN}(t_{k+1}) = DNN(w_{DNN}(t_{k}),I_M(u^i(t_k))$
\STATE Observe $I_M(u^i(t_{k+1}))$ 
\STATE $k=k+1$
\ENDWHILE
\end{algorithmic}
\end{algorithm}

\begin{flushleft}
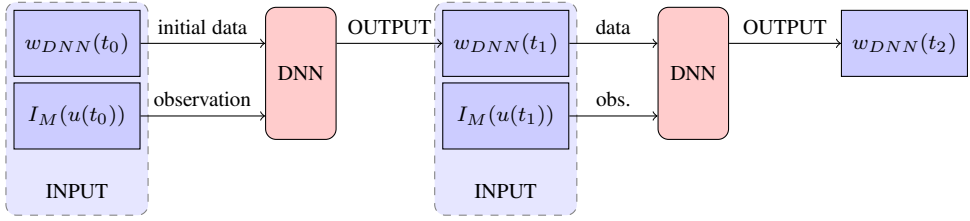
\begin{figure}
\begin{tikzpicture}
    \node (naveq) [naveqs] {\footnotesize DNN};
    \node (naveq1) [naveqs,right of=naveq, xshift=12em] {\footnotesize DNN};
    \path (naveq.140)+(-\blockdist,0) node (gyros) [sensor] {\footnotesize $w_{DNN}(t_0)$};
    \path (naveq.140)+(8.5,0) node (gyros2) [sensor] {\footnotesize $w_{DNN}(t_2)$};
    \path (naveq.140)+(3.2,0) node (gyros1) [sensor] {\footnotesize $w_{DNN}(t_1)$};
    \path (naveq.-130)+(3.2,0) node (accel1) [sensor] {\footnotesize $I_M(u(t_1))$};
    \path (naveq.-130)+(-\blockdist,0) node (accel) [sensor] {\footnotesize $I_M(u(t_0))$};
    \path [draw, ->] (gyros) -- node [above] {\footnotesize initial data} 
        (naveq.west |- gyros) ;
    \path [draw, ->] (gyros1) -- node [above] {\footnotesize data} 
        (naveq1.west |- gyros1) ;
    \path [draw, ->] (accel) -- node [above] {\footnotesize observation} 
        (naveq.west |- accel);
    \path [draw, ->] (accel1) -- node [above] {\footnotesize obs.} 
    (naveq1.west |- accel1);
    \node (IMU) [below of=accel] {\footnotesize INPUT};
    \node (IMU1) [below of=accel1] {\footnotesize INPUT};
    \path (naveq.south west)+(-0.6,-0.4) node (INS) {};
    \path [draw, <-] (gyros1) -- node [above] {\footnotesize OUTPUT} 
    (naveq.east |- gyros1) ;
    \path [draw, <-] (gyros2) -- node [above] {\footnotesize OUTPUT} 
    (naveq1.east |- gyros2) ;
    \begin{pgfonlayer}{background}
        \path (gyros.north west)+(-0.1,0.1) node (a) {};
        \path (IMU.south -| gyros.east)+(+0.1,-0.1) node (b) {};
        \path[fill=blue!10,rounded corners, draw=black!50, dashed]
            (a) rectangle (b);
        \path (gyros1.north west)+(-0.1,0.1) node (a) {};
        \path (IMU1.south -| gyros1.east)+(+0.1,-0.1) node (b) {};
        \path[fill=blue!10,rounded corners, draw=black!50, dashed]
            (a) rectangle (b);
    \end{pgfonlayer}
\end{tikzpicture}
\caption{Two data assimilation steps using the trained DNN with arbitrary initial data $w_{DNN}(t_0)$.}
\label{diag:online}
\end{figure}
\end{flushleft}

\section{Error Analysis}
\label{sec:error}

Using Algorithm~\ref{alg:online}, we have replaced the nudging algorithm by the DNN, where the training data is generated by nudging. Thus to derive the error estimates for the proposed DNN approach, we will first derive the error estimates for the nudging algorithm. Such nudging estimates are problem dependent -- we will proceed with two particular examples. In the first example, we will focus on the widely used Lorenz 63 ODE model. In the second example, we will consider the Navier-Stokes partial differential equation. We will only provide references to the the error estimates in the latter case, for instance \cite{biswas}. Subsequently, overall error estimates for the DNN based Algorithm~\ref{alg:online} will be discussed in subsection~\ref{s:total}.

\subsection{Lorenz 63 and Nudging}
\label{s:lorenz}

The goal of this section is provide exponential type approximation estimates between the nudged (continuous as well as discrete) and the original solution to Lorenz 63 model. The latter is given by the following ODEs for which we assume $\sigma,\rho,\beta > 0$: 
\begin{subequations}
\begin{align}\label{lorenz}
    \frac{dx}{dt} &= \sigma(y-x),\\
    \frac{dy}{dt} &= x(\rho-z)-y, \label{lorenz2}\\ 
    \frac{dz}{dt} &= xy-\beta z \label{lorenz3}. 
\end{align}
\end{subequations}
Notice that, the initial conditions $x(0), y(0)$, and $z(0)$ are unknowns. We define the nudging algorithm for continuous in time $x$-component observations with the altered ODEs:
\begin{align}\label{lorenz_na}
\begin{aligned}
    \frac{d\bar{x}}{dt} &= \sigma(\bar{y}-\bar{x})-\mu(\bar{x}-x),\quad\bar{x}(0)=0,\\ 
    \frac{d\bar{y}}{dt} &= \bar{x}(\rho-\bar{z})-\bar{y},\quad\bar{y}(0)=0,\\ 
    \frac{d\bar{z}}{dt} &= \bar{x}\bar{y}-\beta \bar{z},\quad\bar{z}(0)=0 . 
\end{aligned}    
\end{align}
Here, $\mu > 0$ is the nudging parameter and for simplicity we have assumed that the initial conditions for the nudged variables are 0. Throughout the analysis we assume that our solution $(x,y,z)$ lies on the global attractor giving us the following bounds found in \cite{doering}.
\begin{theorem}\label{thm:kbound}
Let (x,y,z) be a trajectory in the attractor of the Lorenz system \eqref{lorenz}--\eqref{lorenz3} then $x(t)^2+y(t)^2+z(t)^2\le K\quad \forall t\in\mathbb{R}$ where
\begin{equation}\label{k_bound}
    K=\frac{\beta^2(\rho + \sigma)^2}{4(\beta-1)}\approx 1540.27.
\end{equation}
\end{theorem}

We begin by defining the ODEs that govern the evolution of $\Delta X,\ \Delta Y\text{ and }\Delta Z$, let 
\begin{equation}
    \Delta X = x - \bar{x},\quad \Delta Y = y - \bar{y},\quad \Delta Z = z -\bar{z} .
\end{equation}
It is not difficult to see that (cf.~\cite{martinez}):
\begin{align}\label{lorenz_diff}
    \Delta\dot{X} &= -\sigma\Delta X + \sigma\Delta Y-\mu\Delta X ,\\ \nonumber
    \Delta\dot{Y} &= \rho\Delta X -\Delta Y - \bar{x}\Delta Z - \Delta X z,\\ \nonumber
    \Delta\dot{Z} &= \bar{x}\Delta Y + y\Delta X - \beta\Delta Z.\label{lorenz_diff3}
\end{align}
Next, we show that the nudged solution converges in time exponentially to the reference solution. 
\begin{theorem}\label{thm-x-continuous}
Let $V = \Delta X^2 + \Delta Y^2 + \Delta Z^2$, and
$\mu\ge \max\{2,\frac{1}{2} + (\rho+\sigma)^2 - \sigma + K +\frac{K}{2\beta}\}$ where $K$ is given in (\ref{k_bound}). Then
\begin{equation*}
    V(t)\le e^{-ct}V(0)\quad \mbox{for all } t\ge 0.
\end{equation*}
where $c=\min\{1,\beta \}$.
\end{theorem}
\begin{proof}
Multiply the equations for $\Delta \dot{X},\Delta \dot{Y}\text{ and }\Delta \dot{Z}$ by $\Delta X, \Delta Y \text{ and } \Delta Z$ respectively. 
\begin{align*}
    \Delta\dot{X}\Delta X &= -(\sigma+\mu)\Delta X^2 + \sigma\Delta X\Delta Y ,\\ \nonumber
    \Delta\dot{Y}\Delta Y &= \rho\Delta X\Delta Y -\Delta Y^2 - \bar{x}\Delta Y\Delta Z - z\Delta X\Delta Y,\\ \nonumber
    \Delta\dot{Z}\Delta Z &= \bar{x}\Delta Y\Delta Z + y\Delta X\Delta Z - \beta\Delta Z^2. 
\end{align*}
After summing,
\begin{align}\label{lorenz_diff_b}
    \frac{1}{2}\dot{V} + (\sigma + \mu)\Delta X^2 + \Delta Y^2 + \beta\Delta Z^2 = (\rho+\sigma)\Delta X\Delta Y - z\Delta X\Delta Y + y\Delta X\Delta Z.
\end{align}
Each term on the right hand side is then estimated using the Young's inequality and (\ref{k_bound}).
\begin{equation*}
\begin{aligned}
    |(\rho+\sigma)\Delta X\Delta Y| &\le (\rho+\sigma)^2\Delta X^2 + \frac{1}{4}\Delta Y^2 , \\ 
    |z\Delta X\Delta Y| &\le K\Delta X^2 + \frac{1}{4}\Delta Y^2 , \\
    |y\Delta X\Delta Z| &\le \frac{K}{2\beta}\Delta X^2 + \frac{\beta}{2}\Delta Z^2.
\end{aligned}    
\end{equation*}
Combing these estimates with \eqref{lorenz_diff_b} yields,
\begin{equation}
    \frac{1}{2}\dot{V} + \Big(\sigma+\mu -(\rho+\sigma)^2 - K - \frac{K}{2\beta}\Big)\Delta X^2 +\frac{1}{2}\Delta Y^2 + \frac{\beta}{2}\Delta Z^2 \le 0.
\end{equation}
We then choose $\mu$ such that $\sigma+\mu -(\rho+\sigma)^2 - K - \frac{K}{2\beta} \ge \frac{1}{2} $ resulting in the condition $\mu\ge \frac{1}{2} + (\rho+\sigma)^2 - \sigma + K +\frac{K}{2\beta}$. 
Let $c=\min\{1,\beta \}$. Therefore,
\begin{equation}
    \dot{V} + cV \le 0.
\end{equation}
The asserted result then follows immediately.
\end{proof}
Next, we show a similar result in the discrete case. We define the nudging algorithm for discrete in time $x$-component observations with the altered ODEs:
\begin{align}\label{lorenz_na_discrete}
    \frac{d\bar{x}}{dt} &= \sigma(\bar{y}-\bar{x})-\mu(\bar{x}(t_n)-x(t_n)),\quad\forall t\in[t_n,t_{n+1}),\\ \nonumber
    \frac{d\bar{y}}{dt} &= \bar{x}(\rho-\bar{z})-\bar{y},\\ \nonumber
    \frac{d\bar{z}}{dt} &= \bar{x}\bar{y}-\beta \bar{z}, 
\end{align}
 and we define $\delta=t_{n+1}-t_n$. The difference equations now become,
\begin{align}\label{lorenz_diff_discrete}
    \Delta\dot{X} &= -\sigma\Delta X + \sigma\Delta Y -\mu\Delta X(t_n),\\
    \Delta\dot{Y} &= \rho\Delta X -\Delta Y - \bar{x}\Delta Z - \Delta X z,\quad \forall t\in[t_n,t_{n+1}) ,\nonumber\\
    \Delta\dot{Z} &= \bar{x}\Delta Y + y\Delta X - \beta\Delta Z.\nonumber
\end{align}

We will then show that the nudged solution in this case converges in time exponentially to the reference solution. The presented approach is motivated by \cite{foias2016discrete}.
\begin{theorem}\label{thm-x-discrete}
Let $\mu\ge 2(\rho+\sigma)^2 - 2\sigma + 2K +\frac{K}{\beta}$ where $K$ is given in (\ref{k_bound}) and
\begin{equation}\label{delta}
    \delta \le \min\bigg\{\frac{1}{2\mu},\frac{1}{64(\sigma+\mu)^2},\frac{1}{32\mu\sigma^2}\bigg\}.
\end{equation}
Then for $c=\min\{\frac{\mu}{2},1,\beta \}$ and $\gamma = \frac{1}{2c}(1+(2c-1)e^{-c\delta})<1$, we have
\begin{equation}
    V(t)\le \gamma^nV(0),\quad \mbox{for all } t\in[t_n,t_{n+1}).
\end{equation}
\end{theorem}
\begin{proof}
Multiply the equations for $\Delta \dot{X},\Delta \dot{Y}\text{ and }\Delta \dot{Z}$ by $\Delta X, \Delta Y \text{ and } \Delta Z$ respectively. We obtain
\begin{align}\label{lorenz_diff_discrete_a}
    \Delta\dot{X}\Delta X &= -(\sigma+\mu)\Delta X^2 + \sigma\Delta X\Delta Y -\mu(\Delta X(t_n)-\Delta X)\Delta X ,\\ \nonumber
    \Delta\dot{Y}\Delta Y &= \rho\Delta X\Delta Y -\Delta Y^2 - \bar{x}\Delta Y\Delta Z - z\Delta X\Delta Y,\\ \nonumber
    \Delta\dot{Z}\Delta Z &= \bar{x}\Delta Y\Delta Z + y\Delta X\Delta Z - \beta\Delta Z^2. 
\end{align}
After summing, 
\begin{align}\label{lorenz_diff_discrete_b}
    \frac{1}{2}\dot{V} + (\sigma + \mu)\Delta X^2 + \Delta Y^2 + \beta\Delta Z^2 &= (\rho+\sigma)\Delta X\Delta Y - z\Delta X\Delta Y + y\Delta X\Delta Z\\
    &-\mu(\Delta X(t_n) - \Delta X)\Delta X.\nonumber
\end{align}
The first three terms on the right hand side are then estimated using Young's inequality and \eqref{k_bound}, i.e., 
\begin{equation*}
\begin{aligned}
    |(\rho+\sigma)\Delta X\Delta Y|\le (\rho+\sigma)^2\Delta X^2 + \frac{1}{4}\Delta Y^2 , \\
    |z\Delta X\Delta Y|\le K\Delta X^2 + \frac{1}{4}\Delta Y^2 , \\
    |y\Delta X\Delta Z| \le \frac{K}{2\beta}\Delta X^2 + \frac{\beta}{2}\Delta Z^2.
\end{aligned}
\end{equation*}
Moreover, the final term is estimated as follows:
\begin{align}\label{tough_part}
    \mu|(\Delta X(t_n) - \Delta X)\Delta X| &\le \mu \bigg(\int_{t_n}^{t}|\Delta\dot{X}(s)| ds\bigg)|\Delta X|\\
    &\le \mu\bigg(\int_{t_n}^{t} |\Delta\dot{X}(s)| ds\bigg)^2 + \frac{\mu}{4}\Delta X^2.\nonumber
\end{align}
We bound the integral term present on the right-hand-side of (\ref{tough_part}) before continuing. Using (\ref{lorenz_diff_discrete}), 
\begin{equation}
    |\Delta\dot{X}(s)|\le \sigma|\Delta X| + \sigma|\Delta Y| + \mu|\Delta X| + \mu\int_{t_n}^s|\Delta\dot{X}(\tau)|d\tau.
\end{equation}
Integrating both sides we obtain that
\begin{align*}
    \int_{t_n}^t |\Delta \dot{X}(s)|ds &\le \int_{t_n}^t((\sigma+\mu)|\Delta X| + \sigma|\Delta Y|)ds + \mu\delta\int_{t_n}^t|\Delta\dot{X}(s)|ds.
\end{align*}    
Using $\delta \le \frac{1}{2\mu}$, and using H\"older's inequality, we arrive at
\begin{align*}    
    \int_{t_n}^t |\Delta \dot{X}(s)|ds
    &\le 2\int_{t_n}^t((\sigma+\mu)|\Delta X| + \sigma|\Delta Y|)ds \nonumber \\
    &\le 2^{3/2}\delta^{1/2}\bigg(\int_{t_n}^t ((\sigma+\mu)^2\Delta X^2 + \sigma^2\Delta Y^2)ds\bigg)^{1/2} .
\end{align*}
Combining the above derived estimates with \eqref{lorenz_diff_discrete_b} and choosing $\mu\ge 2(\rho+\sigma)^2 - 2\sigma + 2K +\frac{K}{\beta}$,
\begin{equation}\label{lorenz_diff_discrete_c}
    \dot{V} + \frac{\mu}{2}\Delta X^2 + \Delta Y^2 + \beta\Delta Z^2\le 16\mu\delta\int_{t_n}^t ((\sigma+\mu)^2\Delta X^2 + \sigma^2\Delta Y^2)ds \quad =: \boxed{\mathrm{I}}.
\end{equation}
Integrating both sides,
\begin{equation*}
    V(t) - V(t_n) + \int_{t_n}^t\big(\frac{\mu}{2}\Delta X^2 +\Delta Y^2 + \beta\Delta Z^2\big)ds\le 16\mu\delta^2\int_{t_n}^t ((\sigma+\mu)^2\Delta X^2 + \sigma^2\Delta Y^2)ds.
\end{equation*}
Notice that $0 \le V(t)$ by construction. Using this and after rearranging terms, we arrive at 
\begin{equation}\label{lorenz_diff_discrete_cc}
   \int_{t_n}^t\left( \left(\frac{\mu}{2} - 16\mu\delta^2 (\sigma+\mu)^2 \right)\Delta X^2 + (1-16\mu\delta^2\sigma^2)\Delta Y^2 + \beta\Delta Z^2\right)ds \le V(t_n)
\end{equation}
In order to get the final estimates, we will bound (from below) the left-hand-side in \eqref{lorenz_diff_discrete_cc} by $\boxed{\mathrm{I}}$. This can be done by appropriately choosing $\delta$. In particular $\delta \le \frac{1}{64(\sigma+\mu)^2}$ and $\delta \le \frac{1}{32\mu \sigma^2}$, leads to 
\begin{align*}
    \boxed{\mathrm{I}} \le \frac{1}{2}V(t_n).
\end{align*}

Let $c=\min\{\frac{\mu}{2},1,\beta \}$. Therefore (\ref{lorenz_diff_discrete_c}) becomes
\begin{equation}\label{lorenz_diff_discrete_d}
    \dot{V} + cV \le \frac{1}{2}V(t_n),\quad \forall t\in[t_n,t_{n+1}).
\end{equation}
and we conclude 
\begin{align*}
    V(t_{n+1}) &\le \frac{1}{2c}V(t_n)(1-e^{-c\delta}) + e^{-c\delta}V(t_n),\\
    &\le \gamma V(t_n)\le\gamma^{n+1} V(t_0),\quad \gamma :=\frac{1}{2c}(1+(2c-1)e^{-c\delta})<1.
\end{align*}
Recall $\delta=t_{n+1}-t_n$. Integrating (\ref{lorenz_diff_discrete_d}) over $[t_n,t]$ yields
\begin{equation}
    V(t) \le \frac{1}{2c}V(t_n)(1+(2c-1)e^{-(t-t_n)})\le V(t_n)\le \gamma^n V(t_0)\quad\text{for } t_n<t<t_{n+1}.
\end{equation}
The proof is complete.
\end{proof}

Similar proofs will show exponential convergence for observations taken in any two or three components for both continuous and discrete in time observations. The continuous case proof becomes easier with more observations while the discrete case proof becomes more cumbersome. We provide the details for the case of discrete in time observations in the $y$ and $z$ components. The major difference in proving Theorem \ref{thm-x-discrete} with observations in two or three components is deriving (\ref{tough_part}) for each of the components with observations as this requires us to have a bound on the nudged solution $\bar{w}$. 

We define the nudging algorithm for discrete in time $y,z$-component observations with the altered ODEs:
\begin{align}\label{lorenz_na_discrete_xy}
    \frac{d\bar{x}}{dt} &= \sigma(\bar{y}-\bar{x}),\\ \nonumber
    \frac{d\bar{y}}{dt} &= \bar{x}(\rho-\bar{z})-\bar{y}-\mu(\bar{y}(t_n)-y(t_n)),\quad\forall t\in[t_n,t_{n+1}),\\ \nonumber
    \frac{d\bar{z}}{dt} &= \bar{x}\bar{y}-\beta \bar{z}-\mu(\bar{z}(t_n)-z(t_n)), 
\end{align}
 and we define $\delta=t_{n+1}-t_n$. The difference equations now become,
\begin{align}\label{lorenz_diff_discrete_yz}
    \Delta\dot{X} &= -\sigma\Delta X + \sigma\Delta Y ,\\
    \Delta\dot{Y} &= \rho\Delta X -\Delta Y - \bar{x}\Delta Z - \Delta X z-\mu\Delta Y(t_n),\quad \forall t\in[t_n,t_{n+1}) ,\nonumber\\
    \Delta\dot{Z} &= \bar{x}\Delta Y + y\Delta X - \beta\Delta Z-\mu\Delta Z(t_n).\nonumber
\end{align}
We will then show that the nudged solution in this case converges in time exponentially to the reference solution on a finite interval provided we have a bound on the nudged solution on this finite interval. Subsequently, we will prove that the nudged solution is in fact globally bounded.
\begin{theorem}\label{thm-yz-discrete}
Assume $|\bar{w}(t)|^2\le \widetilde{K}\quad\forall t\in[0,t_K]$. Let $\mu\ge 4\max\{\frac{(\rho+\sigma)^2}{\sigma} + K -1, K-\beta\}$ where $K$ is given in (\ref{k_bound}), $\widetilde{K}=5K$ and
\begin{equation}\label{delta_yz}
    \delta \le \frac{1}{64}\min\bigg\{\frac{\sigma}{2\mu((\rho+K^{\frac{1}{2}})^2+K)},\frac{1}{((1+\mu)^2+\widetilde{K})},\frac{1}{((\beta+\mu)^2+\widetilde{K})}\bigg\}.
\end{equation}
Furthermore, $\{t_n\}$ satisfy $t_n\le t_K$.
Then for $c=\min\{\frac{\sigma}{2},\mu\}$ and $\gamma = \frac{1}{2c}(1+(2c-1)e^{-c\delta})<1$
the following holds
\begin{equation}
    V(t)\le \gamma^nV(0)\quad \forall t\in[t_n,t_{n+1}),\ .
\end{equation}
\end{theorem}
\begin{proof}
Multiply the equations for $\Delta \dot{X},\Delta \dot{Y}\text{ and }\Delta \dot{Z}$ by $\Delta X, \Delta Y \text{ and } \Delta Z$ respectively. 
\begin{align}\label{lorenz_diff_discrete_a_yz}
    \Delta\dot{X}\Delta X &= -\sigma\Delta X^2 + \sigma\Delta X\Delta Y  ,\\ \nonumber
    \Delta\dot{Y}\Delta Y &= \rho\Delta X\Delta Y -\Delta Y^2 - \bar{x}\Delta Y\Delta Z - z\Delta X\Delta Y-\mu\Delta Y(t_n)\Delta Y,\\ \nonumber
    \Delta\dot{Z}\Delta Z &= \bar{x}\Delta Y\Delta Z + y\Delta X\Delta Z - \beta\Delta Z^2-\mu\Delta Z(t_n)\Delta Z. 
\end{align}
After summing,
\begin{align}\label{lorenz_diff_discrete_b_yz}
    \frac{1}{2}\dot{V} + \sigma\Delta X^2 + (1+\mu)\Delta Y^2 &+ (\beta+\mu)\Delta Z^2 = (\rho+\sigma)\Delta X\Delta Y - z\Delta X\Delta Y + y\Delta X\Delta Z\\
    &-\mu(\Delta Y(t_n) - \Delta Y)\Delta Y-\mu(\Delta Z(t_n) - \Delta Z)\Delta Z.\nonumber
\end{align}
Each term on the right hand side is then estimated using Young's inequality and (\ref{k_bound}).
\begin{equation}\label{eq:thmyz1}
    |(\rho+\sigma)\Delta X\Delta Y|\le \frac{\sigma}{4}\Delta X^2 + \frac{(\rho+\sigma)^2}{\sigma}\Delta Y^2. 
\end{equation}
\begin{equation}\label{eq:thmyz2}
    |z\Delta X\Delta Y|\le \frac{\sigma}{4}\Delta X^2 + K\Delta Y^2.
\end{equation}
\begin{equation}\label{eq:thmyz3}
    |y\Delta X\Delta Z| \le \frac{\sigma}{4}\Delta X^2 + K\Delta Z^2.
\end{equation}
Next we estimate the terms arising from the nudging term.
\begin{align}\label{tough_part_y}
    \mu|(\Delta Y(t_n) - \Delta Y)\Delta Y| &\le \mu\bigg(\int_{t_n}^{t}|\Delta\dot{Y}(s)|ds\bigg)|\Delta Y|\\
    &\le \mu\bigg(\int_{t_n}^{t}|\Delta\dot{Y}(s)|ds\bigg)^2 + \frac{\mu}{4}\Delta Y^2.\nonumber
\end{align}
In the same vein, we have that
\begin{align}\label{tough_part_z}
    \mu|(\Delta Z(t_n) - \Delta Z)\Delta Z| &\le \mu\bigg(\int_{t_n}^{t}|\Delta\dot{Z}(s)|ds\bigg)|\Delta Z|\\
    &\le \mu\bigg(\int_{t_n}^{t}|\Delta\dot{Z}(s)|ds\bigg)^2 + \frac{\mu}{4}\Delta Z^2.\nonumber
\end{align}
We bound the integral term present in the right hand side of (\ref{tough_part_y}) and (\ref{tough_part_z})  before continuing. Using (\ref{lorenz_diff_discrete_yz}), 
\begin{equation}\label{K_tilde}
    |\Delta\dot{Y}(s)|\le (\rho + K^{1/2})|\Delta X| + (1+\mu)|\Delta Y| + \widetilde{K}^{1/2}|\Delta Z| + \mu\int_{t_n}^s|\Delta\dot{Y}(\tau)|d\tau.
\end{equation}
\begin{equation}\label{eq:ZZ}
    |\Delta\dot{Z}(s)|\le K^{1/2}|\Delta X| + \widetilde{K}^{1/2}|\Delta Y| + (\beta+\mu)|\Delta Z| + \mu\int_{t_n}^s|\Delta\dot{Z}(\tau)|d\tau.
\end{equation}
Integrating both sides in \eqref{K_tilde}, using condition \eqref{delta} and applying Hölder's inequality, we obtain 
\begin{equation*}
    \int_{t_n}^t |\Delta \dot{Y}(s)|ds \le \int_{t_n}^t((\rho + K^{\frac{1}{2}})|\Delta X| + (1+\mu)|\Delta Y| + \widetilde{K}^{\frac{1}{2}}|\Delta Z|)ds+\mu\delta\int_{t_n}^t|\Delta\dot{Y}(s)|ds 
\end{equation*}    
From which, after using $\mu \delta \le 1/2$, we arrive 
\begin{align}\label{eq:tough_part_y1}    
 \int_{t_n}^t |\Delta \dot{Y}(s)|ds 
    &\le 2\int_{t_n}^t((\rho + K^{\frac{1}{2}})|\Delta X| + (1+\mu)|\Delta Y| + \widetilde{K}^{\frac{1}{2}}|\Delta Z|)ds. \\ \nonumber
    &\le 4\delta^{1/2}\bigg(\int_{t_n}^t \Big((\rho + K^{\frac{1}{2}})^2|\Delta X|^2 + (1+\mu)^2|\Delta Y|^2 + \widetilde{K}|\Delta Z|^2 \Big)ds\bigg)^{1/2} ,
\end{align}
where in the last step we have used H\"older's and Young's inequalities. Similarly, from \eqref{eq:ZZ} we obtain
\begin{align}\label{eq:tough_part_z1}
    \int_{t_n}^t |\Delta \dot{Z}(s)|ds 
    &\le 4\delta^{\frac12}\bigg(\int_{t_n}^t (K|\Delta X|^2 + \widetilde{K}|\Delta Y|^2 + (\beta+\mu)^2|\Delta Z|^2)ds\bigg)^{1/2} .
\end{align}
Using the estimates from \eqref{eq:thmyz1} -- \eqref{tough_part_z} in \eqref{lorenz_diff_discrete_b_yz}, we obtain that 
\begin{align*}
    &\frac12 \dot{V} + \frac{\sigma}{4}\Delta X^2 + \left(1+\frac{3\mu}{4} - \frac{(\rho+\sigma)^2}{\sigma} - K\right)\Delta Y^2 + \left( \beta + \frac{3\mu}{4} - K \right)\Delta Z^2 \\
    &\le \mu \left( \int_{t_n}^t |\Delta \dot{Y}| \right)^2 
       + \mu \left( \int_{t_n}^t |\Delta \dot{Z}| \right)^2 .
\end{align*}    
Subsequently, using \eqref{eq:tough_part_y1}, and \eqref{eq:tough_part_z1} and using the condition on $\mu$, we arrive at
\begin{align}\label{lorenz_diff_discrete_c_yz}
    &\dot{V} + \frac{\sigma}{2}\Delta X^2 + \mu\Delta Y^2 + \mu\Delta Z^2\\ \nonumber
    &\le 32\mu\delta\int_{t_n}^t \Big((\rho+K^{\frac{1}{2}})^2+K)\Delta X^2 + ((1+\mu)^2+\widetilde{K})\Delta Y^2 +((\beta+\mu)^2+\widetilde{K})\Delta Z^2 \Big)ds\\ \nonumber
    &=: \boxed{\mathrm{I}}.
\end{align}
Integrating both sides of \eqref{lorenz_diff_discrete_c_yz},
\begin{align}
    &V(t) - V(t_n) + \int_{t_n}^t\big(\frac{\sigma}{2}\Delta X^2 + \mu\Delta Y^2 + \mu\Delta Z^2\big)ds \\ \nonumber
   &\le 32\mu\delta^2\int_{t_n}^t ((\rho+K^{\frac{1}{2}})^2+K)\Delta X^2 + ((1+\mu)^2+\widetilde{K})\Delta Y^2 +((\beta+\mu)^2+\widetilde{K})\Delta Z^2)ds
\end{align}
Notice that $0 \le V(t)$ by construction. Using this and after rearranging terms, we arrive at 
\begin{align}
    \label{lorenz_diff_discrete_yz_c}
    &\int_{t_n}^t \big((\frac{\sigma}{2} - 32\mu\delta^2((\rho+K^{\frac{1}{2}})^2+K))\Delta X^2 +  \big(\mu-32\mu\delta^2((1+\mu)^2+\widetilde{K}))\Delta Y^2\\ \nonumber
    &\qquad + (\mu-32\mu\delta^2((\beta+\mu)^2+\widetilde{K}))\Delta Z^2\big)ds\\ 
    &\le V(t_n). \nonumber
\end{align}
In order to get the final estimates, we will bound (from below) the left-hand-side in \eqref{lorenz_diff_discrete_yz_c} by $\boxed{\mathrm{I}}$. This can be done by appropriately choosing $\delta$. In particular \ref{delta_yz} leads to
\begin{align*}
    \boxed{\mathrm{I}} \le \frac{1}{2}V(t_n).
\end{align*}
Let $c=\min\{\frac{\sigma}{2},\mu\}$. Therefore (\ref{lorenz_diff_discrete_c_yz}) becomes
\begin{equation}\label{lorenz_diff_discrete_d_yz}
    \dot{V} + cV \le \frac{1}{2}V(t_n),\quad \forall t\in[t_n,t_{n+1}) , 
\end{equation}
and we conclude 
\begin{align*}
    V(t_{n+1}) &\le \frac{1}{2c}V(t_n)(1-e^{-c\delta}) + e^{-c\delta}V(t_n),\\
    &\le \gamma V(t_n)\le\gamma^{n+1} V(t_0),\quad \gamma :=\frac{1}{2c}(1+(2c-1)e^{-\delta})<1.
\end{align*}
Recall $\delta=t_{n+1}-t_n$. Integrating (\ref{lorenz_diff_discrete_d_yz}) over $[t_n,t]$ yields
\begin{equation}
    V(t) \le \frac{1}{2}V(t_n)(1+e^{-c\delta})\le V(t_n)\le \gamma^n V(t_0)\quad\text{for } t_n<t<t_{n+1}.
\end{equation}
The proof is complete.
\end{proof}
Using the previous theorem we can now prove a bound for the nudged ODEs \eqref{lorenz_na_discrete_xy} which will allow us to prove Theorem \ref{thm-yz-discrete} for $t\in[0,\infty)$.
\begin{theorem}\label{thm:wtilde_bound}
Let $\mu\ge 4\max\{\frac{(\rho+\sigma)^2}{\sigma} + K -1, K-\beta\}$ where $K$ is given in (\ref{k_bound}), $\widetilde{K}=5K$ and $\delta$ satisfies \eqref{delta_yz}.
Then the nudged ODEs \eqref{lorenz_na_discrete_xy} satisfy
\[
|\bar{w}|^2\le \widetilde{K} := 5K\quad\forall t\in[0,\infty).
\]
\end{theorem}
\begin{proof}
Let $w=(x,y,z)$ and $\bar{w}=(\bar{x},\bar{y},\bar{z})$ be the Lorenz and nudging solutions respectively. Assume $|\bar{w}|^2\le 5K=: \widetilde{K}$ doesn't hold $\forall t\in [0,\infty)$ and define 
\[
T_K = \bigg\{s>0: |\bar{w}(t)|^2 \le \widetilde{K}\ \ \forall t\in [0,s] \bigg\}.
\]
$T_K$ is non-empty as $|\bar{w}(0)|=0$ and it is closed as $|\bar{w}|$ is continuous. Define $t^* = \max(T_K)$.  Then  by theorem \ref{thm-yz-discrete} with $t_K = t^*$ we would conclude $V(t^*)\le V(t_0)\le K$. But this is a contradiction since $|\bar{w}(t^*)|\le |w(t^*)-\bar{w}(t^*)| + |w(t^*)|= V(t^*)^{\frac{1}{2}}+|w(t^*)|\le 2K^{1/2}$. Recall $|w(t^*)|^2\le K$ holds by theorem \ref{thm:kbound}. Therefore, $|\bar{w}| \le \widetilde{K}^{\frac12}$.
\end{proof}
Now we can prove Theorem \ref{thm-yz-discrete} for $t\in[0,\infty)$.
\begin{corollary}\label{coro:yz-discrete}
Let $\mu\ge 4\max\{\frac{(\rho+\sigma)^2}{\sigma} + K -1, K-\beta\}$ where $K$ is given in (\ref{k_bound}), $\widetilde{K}=5K$ and $\delta$ satisfies \eqref{delta_yz}.
Then for $c=\min\{\frac{\sigma}{2},\mu\}$, $\gamma = \frac{1}{2c}(1+(2c-1)e^{-c\delta})<1$, \eqref{lorenz_na_discrete_xy} satisfies
\begin{equation}
    V(t)\le \gamma^nV(0)\quad \forall t\in[t_n,t_{n+1}).
\end{equation}
\end{corollary}
\begin{proof}
By Theorem \ref{thm:wtilde_bound}, $|\bar{w}|^2\le \widetilde{K}\quad\forall t\in[0,\infty)$. Using this bound and the exact same proof from Theorem \ref{thm-yz-discrete} we obtain the corollary. 
\end{proof}
\subsection{Navier-Stokes and Nudging}
\label{s:NS}
In the case of the 2D Navier--Stokes equations corresponding results for Theorems \ref{thm-x-continuous}, \ref{thm-x-discrete} and Corollary \ref{coro:yz-discrete} exist. The continuous in time observation case is found in \cite{AOT} and the discrete in time observation case is found in \cite{foias2016discrete}. Both results yield exponential convergence in their respective cases provided appropriate choice of $\mu, \delta$ and appropriate observations are given.

\subsection{DNN Based Data Assimilation}
\label{s:total}
Let $u$ represent the true or reference solution, let $w$ represent the nudging algorithm solution, and let $w_{DNN}$ represent the proposed DNN algorithm. Then by the simple triangle inequality,
\begin{align*}
    \|u-w_{DNN}\|\le \|u-w\| + \|w-w_{DNN}\| := \mathrm{I} + \mathrm{II}.
\end{align*}
The above theory gives us a handle on $\mathrm{I}$. In addition, there is a growing of references that aims to provide bounds on $\mathrm{II}$ for feedforward type DNNs, see \cite{MR4298220} and references therein.

\section{Experimental Protocol}
\label{sec:protocol}

The DNN is trained offline using Algorithm~\ref{alg:offline}. To generate the training data (see~\eqref{trainingdata}), reference solutions are generated from the reference model (\ref{ODE}) using different initial conditions. These synthetic observations occurring every $\delta$ time units are then collected from the reference model solutions and we generate the nudging data (part of the training data) using (\ref{nudging}) on the observations. 

The performance of the algorithm is evaluated in the following experiment: After training the DNN, the reference model is used to produce additional reference solutions which generate synthetic observations for Algorithm~\ref{alg:online}. The DNN based Algorithm~\ref{alg:online} is then given the synthetic observations and we assess the prediction capabilities of the algorithm.

The performance of the algorithm is evaluated by tracking $N^f$ different reference solutions that were not part of the training or validation data. As an evaluation metric we will use the spatio-temporal root mean square error (RMSE):
\begin{equation}\label{RMSE}
RMSE = \sqrt{\frac{1}{(K-k_0)\times N^f}\sum_{k=k_0}^K\sum_{n=1}^{N^f}\big(x_{n,k}^{\text{DA}}-x_{n,k}^{\text{ref}}\big)^2}.
\end{equation}
Here $k_0$ represents the time in which we begin to consider the performance and which starts sometime after the tracking begins in order to remove the effect of the initial state. Let $x_{n,k}^{\text{ref}}$ denotes a reference solution (observations) at time $t_k$. Then $x_{n,k}^{\text{DA}}$ represents the data assimilation solution at time $t_k$ (computed either using nudging or DNN) while using $x_{n,k}^{\text{ref}}$ as the observation data.

\section{Experimental results}
\label{sec:experiments}

The purpose of this section is show the effectively of the proposed algorithm on two well-known ODE examples, i.e., Lorenz 63 and Lorenz 96.

\subsection{Lorenz 63 model}
\label{sec:experiments63}
The Lorenz 63 model is given by the three coupled ODEs \eqref{lorenz}--\eqref{lorenz3}. 
We set $\sigma=10,\ \beta=8/3,\ \rho=28,$ which is known to exhibit chaotic behavior \cite{lorenz}. The equations \eqref{lorenz}--\eqref{lorenz3} are solved using an explicit Runge-Kutta (4,5) in MATLAB. We generate $N_s=1,000$ different initial conditions from three i.i.d. normal distributions with mean $0$ and standard deviation of $10$. The initial conditions are propagated forward to 110 time units using the ODEs \eqref{lorenz}--\eqref{lorenz3} , and we take observations every $\delta=0.1$ time units starting at 100 time units. 

For this example, we will consider two cases, $x$-component and $y$-component observations, respectively. The synthetic observations  are then fed into the nudging algorithm (\ref{nudging}) with $\mu=30$ in the case of $x$ observations and $\mu=10$ in the case of $y$ observations. We generate $N_s = 1,000$ nudging solutions, using the synthetic observations, for a length of 10 time units. Thus, we have $N_s \times \frac{10}{\delta} = 100,000$ samples to choose from. Out of those, we select $N \times N_s$ training samples, where $N = 15$. In particular, we focus on the samples corresponding to the few initial time steps $t = 0, 0.1, 0.2, \dots, 1.5$. As the initial time steps nudging is still tracking the reference solution so these samples will result in a DNN with better forecasting ability compared to a DNN trained on samples where nudging has already converged. The workings of input-output training data generation are described in Algorithm \ref{alg:offline} and Figure~\ref{diagram:offline}.

\subsubsection{Neural Network Architecture} \label{sec:NNA}

Using the setup of \eqref{ResNet1}--\eqref{ResNet4} we train a DNN for each of the components of the model using the 15,000 training samples generated in section \ref{sec:experiments63}. We do the usual 80-20 split, i.e., we use 80\% of the training samples for training and $20\%$ of the training samples for validation. In order to avoid over-fitting we use a patience of 400 iterations with the validation data. If the validation error increases, training continues for an addition number of iterations specified by the patience. Each DNN has 3 hidden layers with a width of 50 each. BFGS optimization routine is used during the training with bias ordering to solve the optimization problem. We use box initialization \cite{box} to initialize the parameters.

\subsubsection{Results} 

For testing, in Figure \ref{fig:sidebyside} we generate two random initial conditions
outside of the usual range of the training data using three i.i.d. standard normal distributions with mean 0 and a standard deviation of 50. We run the initial conditions to 110 time units, and we store observations starting at 100 time units. Using the observations, we calculate the nudging and DNN solutions using Algorithm~\ref{alg:online}.  We see that the DNN is able to track the reference solution. Next we calculate RMSE values \eqref{RMSE} for both data assimilation algorithms with the settings $N^f=100$ and $k_0$ corresponding to 5 time units after observations begin. See table (\ref{table1}) for the results. The results indicate that the DNN was able to recover most of the accuracy of the nudging algorithm.

\begin{figure}[h]
    \centering
    \includegraphics[width=0.495\textwidth]{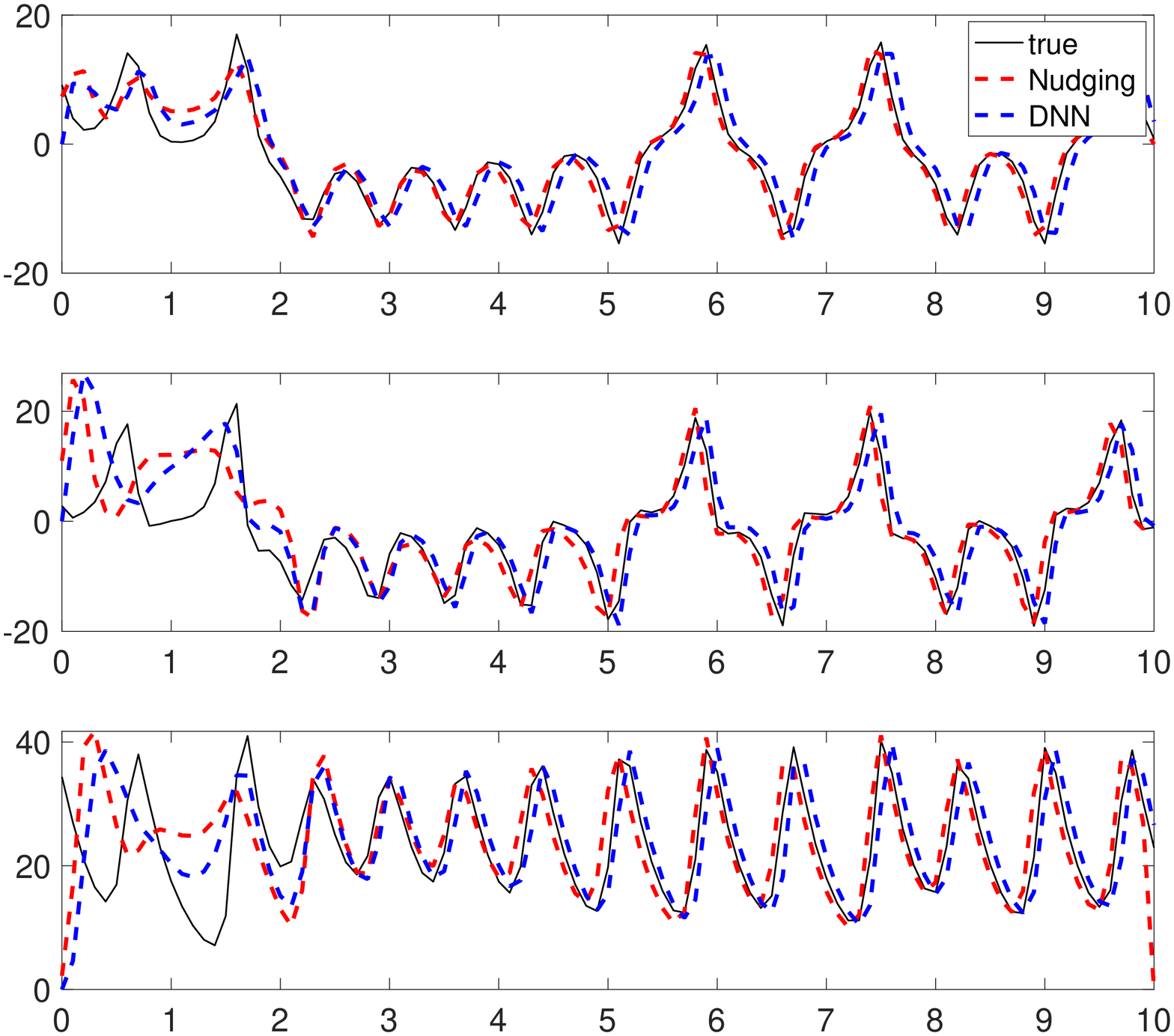}
    \includegraphics[width=0.495\textwidth]{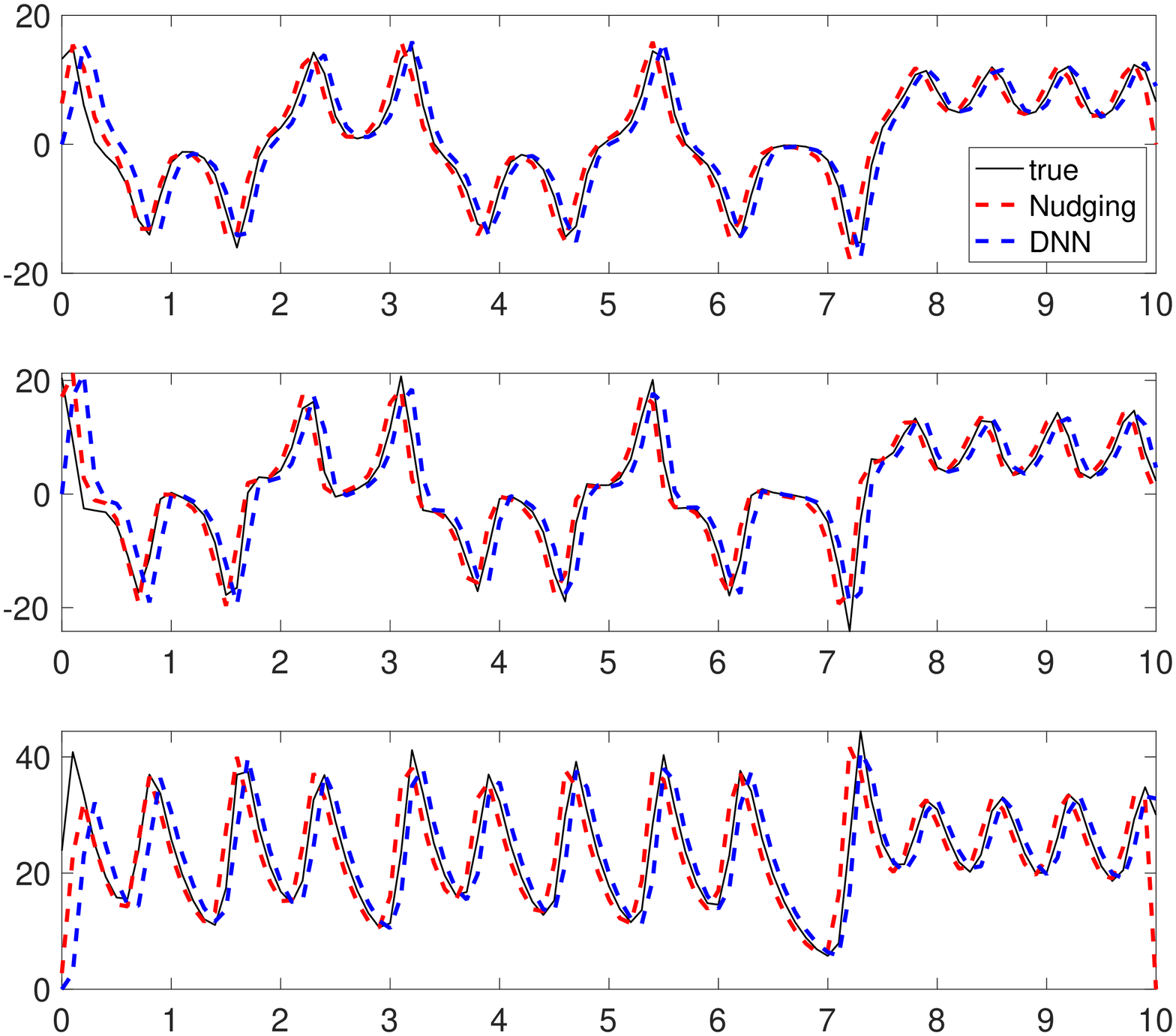}
    \caption{Left: $x$-component observations. Right: $y$-component observations. From top to bottom the $x,y,z$ components of the reference (true), nudging, and DNN solutions over 10 time units for the Lorenz 63 model. }
    \label{fig:sidebyside}
\end{figure}

\begin{table}[!htb]
\begin{center}
\begin{tabular}{ |l|l|l| }
\hline
\multicolumn{3}{ |c| }{RMSE} \\
\hline
 & $x$-obs & $y$-obs \\ \hline
Nudging & 6.0782 & 5.7953 \\
DNN & 6.4456 & 5.8000 \\\hline

\hline
\end{tabular}
\end{center}
\caption{Lorenz 63 RMSE values for nudging and the DNN calculated on 100 reference solutions over the time interval 5 to 10 time units. Notice that the proposed DNN approach can recover similar estimates to nudging.}
\label{table1}
\end{table}
\FloatBarrier
\subsection{Lorenz 96 model}
The Lorenz 96 model is given by the following ODEs: 
\begin{align}\label{Lorenz96}
    \frac{dx_i}{dt} = (x_{i+1} - x_{i-2})x_{i-1} - x_i + F,\quad i=1,2,...,40,\\
    x_{-1} = x_{39}, \quad x_0 = x_{40}, \quad \text{ and }\quad  x_{41} = x_1.
\end{align}
We set $F=10$ which is known to exhibit chaotic behavior \cite{Lorenz96}. We follow a similar procedure as with the Lorenz 63 model and we consider three different observation patterns. We observe approximately $10\%,33\%,\text{and } 50\%$ of the state which corresponds to observing $4,13,\text{and }20$ components, respectively. 
\subsubsection{Neural Network Architecture}
The setup here is similar as in the Lorenz 63 case presented in section \ref{sec:NNA} except for one key difference. We refer to DNN-full as the DNN trained to take in as input the full state in $\mathbb{R}^{40}$ plus the observations. As before, DNN-full is trained on data as described in (\ref{trainingdata}). 
 We refer to DNN-reduced as the DNN that takes in as input a reduced state and reduced observations. DNN-reduced takes advantage of the structure of (\ref{Lorenz96}) and gives the DNN corresponding to component $x_i$, with data from only $\{x_i,x_{i+1},x_{i-1},x_{i-2}\}$. This means the input data, including the observations, for each $i$ is in at most $\mathbb{R}^6$. DNN-reduced trained on 15k samples for the cases of 13 and 20 observations used an architecture of 3 hidden layers each with a width of 50  observations. DNN-reduced trained on 15k samples for the case of 4 observations used an architecture of 15 hidden layers each with a width of 10. DNN-reduced trained on 30k samples used an architecture of 5 hidden layers each with a width of 100 for all cases. 

\subsubsection{Results} 
We proceed as in the case for the Lorenz 63 model. See Figures \ref{fig:sidebyside2}, \ref{fig:lorenz96_4} for the plots and Table \ref{table2} for the RMSE values. We observe in each of the three observation cases the proposed DNN algorithm has a comparable accuracy to the nudging algorithm. In addition, we see from Figure \ref{fig:sidebyside2} that the DNN algorithm is not simply following the nudging solution. The DNN has its own unique solution with the ability to track the true solution.

\begin{table}[!htb]
\begin{center}
\begin{tabular}{ |l|l|l|l| }
\hline
\multicolumn{4}{ |c| }{RMSE} \\
\hline
 & 20 obs & 13 obs & 4 obs \\ \hline
Nudging & 11.9754 & 25.1511 & 36.4937 \\
DNN-reduced(15k) & 17.6243(3x50) & 39.2055(3x50) & 39.7268(15x10)\\
DNN-reduced(30k) & 16.9915(5x100) & 26.2597(5x100) & 37.7265(5x100)\\\hline

\hline
\end{tabular}
\end{center}
\caption{Lorenz 96 RMSE values for nudging and the DNNs calculated on 100 reference solutions over the time interval 5 to 20 time units.}
\label{table2}
\end{table}

\begin{figure}[h]
    \centering
    \includegraphics[width=0.495\textwidth]{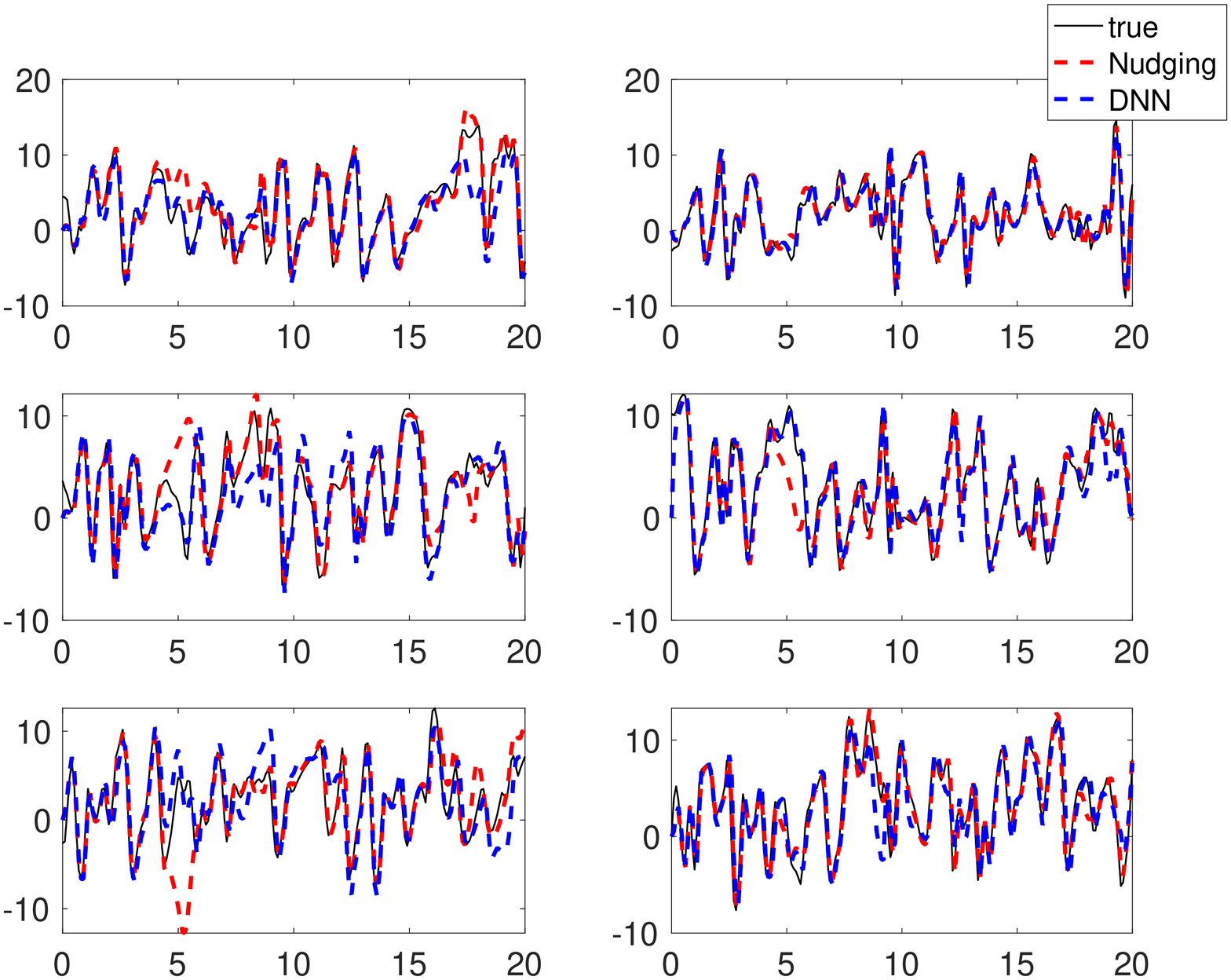}
    \includegraphics[width=0.495\textwidth]{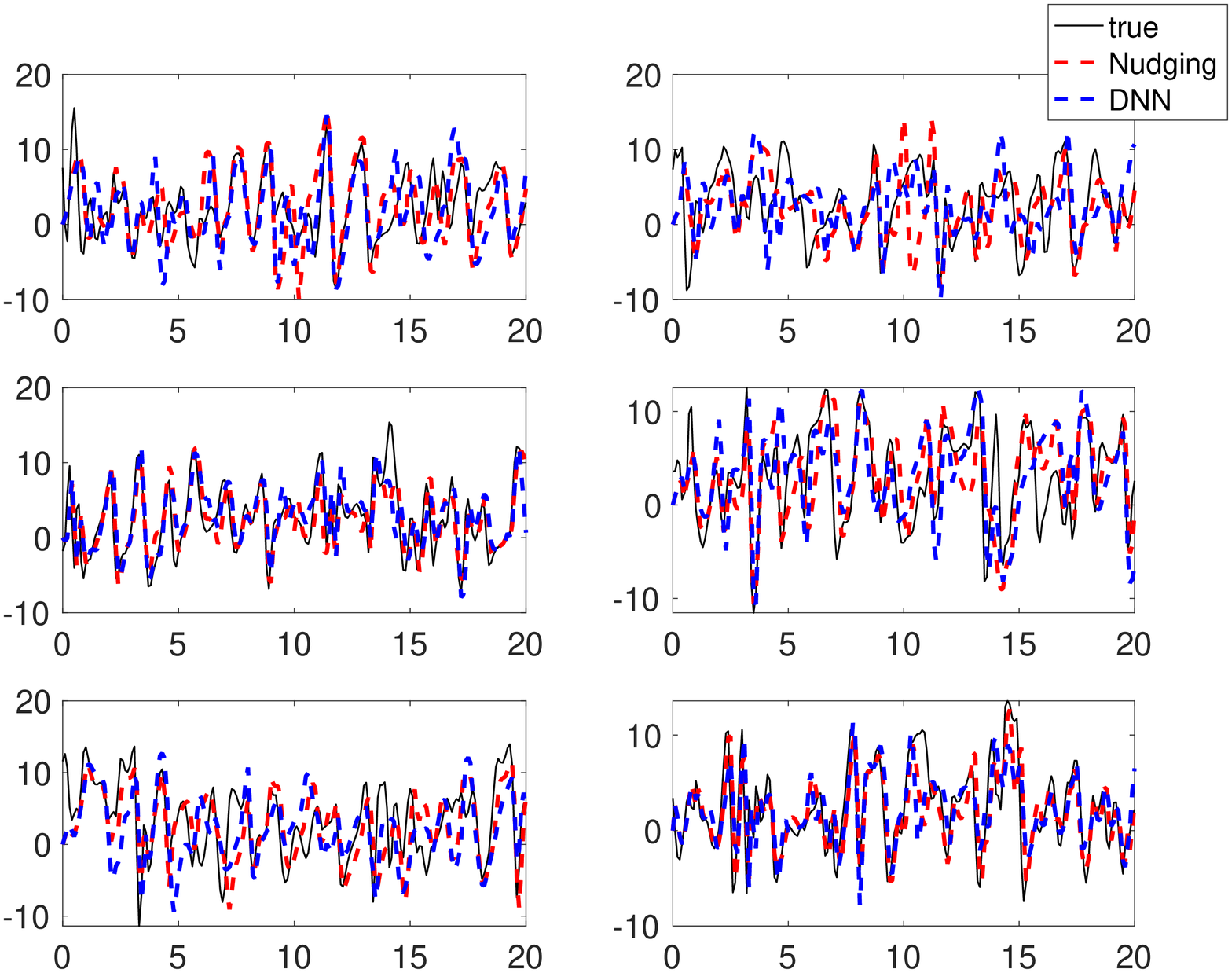}
    \caption{Left: 20 observations on even components. Right: 13 observations on every third component. From left to right, top to bottom the first 6 components of the reference (true), nudging, and DNN solutions over 20 time units for the Lorenz 96 model.}
    \label{fig:sidebyside2}
\end{figure}

\begin{figure}[!htb]
\begin{center}
\includegraphics[width=12cm]{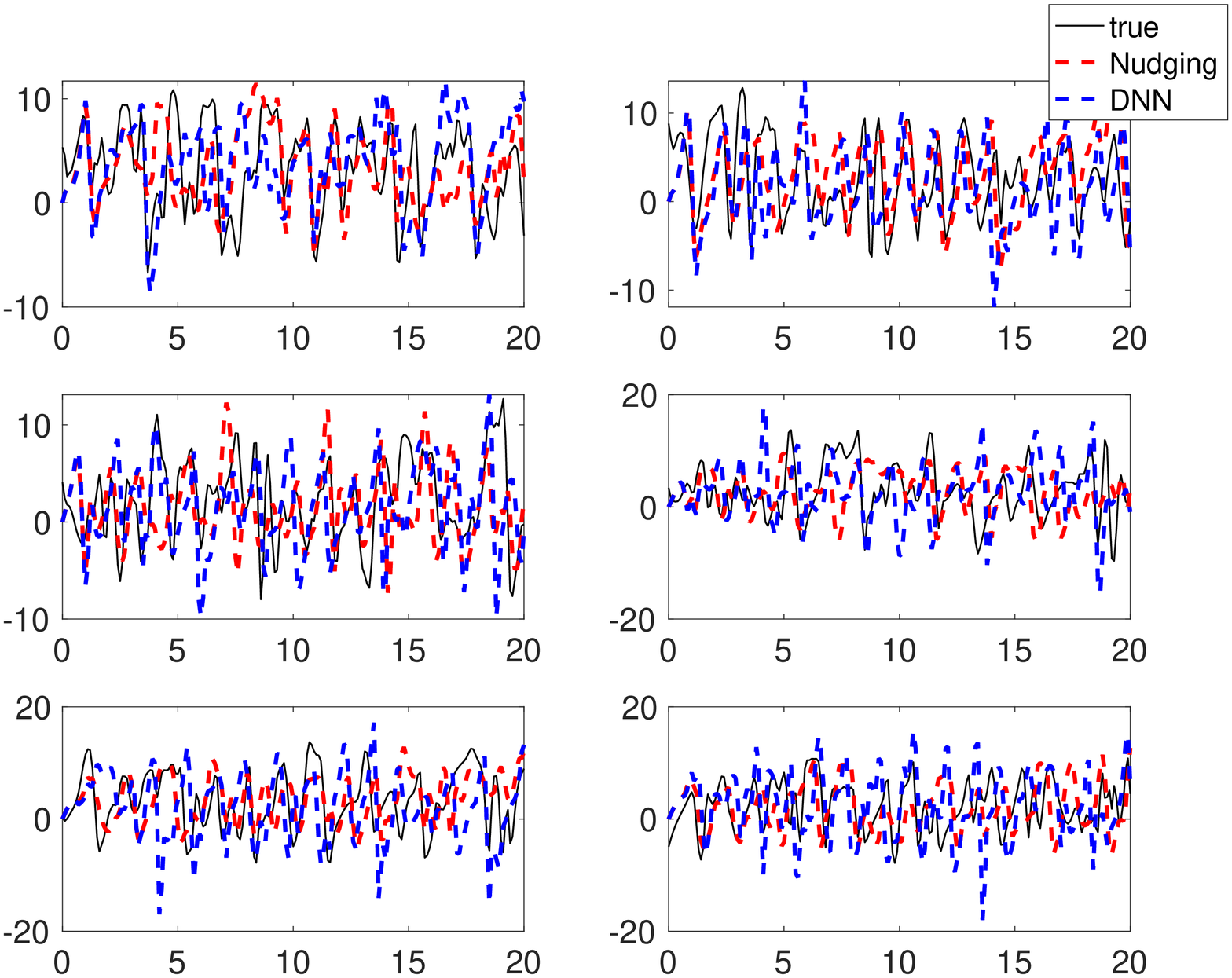}
\end{center}
\caption{From left to right, top to bottom the first 6 components of the reference (true), nudging, and DNN solutions over 20 time units for the Lorenz 96 model with observations every tenth component.}%
\label{fig:lorenz96_4}%
\end{figure}

\section{Conclusions}
\label{sec:conclusions}
A purely neural network based data assimilation algorithm has been introduced. The key idea is to learn the nudging data assimilation algorithm by generating training data that includes observations and the associated nudging algorithm input/output. The result is a DNN that takes as input an initial condition paired with the corresponding observation and the output is able to track the underlying reference solution from which the partial observations are derived. The numerical experiments presented show that the DNN data assimilation algorithm is able to retain most of the accuracy of the nudging algorithm in the case of the Lorenz 63 and Lorenz 96 models. In the case of the Lorenz 96 model, we successfully track the 40 component system with discrete in time observations. This results gives up hope that our algorithm will be useful for more complicated systems such as the Navier-Stokes equations. Finally, we emphasize that within our setup one could easily replace nudging algorithm by other data assimilation algorithms such as the ensemble Kalman filter (EnKF).

\bibliographystyle{siamplain}
\bibliography{references}

\begin{thebibliography}{10}

\bibitem{Albanez2016ContinuousDA}
{\sc D.~A.~F. Albanez, H.~J. Nussenzveig~Lopes, and E.~S. Titi}, {\em
  Continuous data assimilation for the three-dimensional
  {N}avier--{S}tokes-{$\alpha$} model}, Asymptot. Anal., 97 (2016),
  pp.~139--164, \url{https://doi.org/10.3233/ASY-151351},
  \url{https://doi.org/10.3233/ASY-151351}.

\bibitem{DataAssimilationandInitializationofHurricanePredictionModels}
{\sc R.~A. Anthes}, {\em Data assimilation and initialization of hurricane
  prediction models}, Journal of Atmospheric Sciences, 31 (1974), pp.~702 --
  719, \url{https://doi.org/10.1175/1520-0469(1974)031<0702:DAAIOH>2.0.CO;2},
  \url{https://journals.ametsoc.org/view/journals/atsc/31/3/1520-0469_1974_031_0702_daaioh_2_0_co_2.xml}.

\bibitem{antil2021deep}
{\sc H.~Antil, T.~S. Brown, R.~Löhner, F.~Togashi, and D.~Verma}, {\em Deep
  neural nets with fixed bias configuration}, 2021,
  \url{https://arxiv.org/abs/2107.01308}.

\bibitem{Antil2021NovelDN}
{\sc H.~Antil, H.~Elman, A.~Onwunta, and D.~Verma}, {\em Novel deep neural
  networks for solving bayesian statistical inverse}, ArXiv, abs/2102.03974
  (2021).

\bibitem{HAntil_HCElman_AOnwunta_DVerma_2021a}
{\sc H.~Antil, H.~C. Elman, A.~Onwunta, and D.~Verma}, {\em Novel deep neural
  networks for solving bayesian statistical inverse problems}, arXiv preprint
  arXiv:2102.03974,  (2021).

\bibitem{antil2020fractional}
{\sc H.~Antil, R.~Khatri, R.~L{\"o}hner, and D.~Verma}, {\em Fractional deep
  neural network via constrained optimization}, Machine Learning: Science and
  Technology, 2 (2020), p.~015003.

\bibitem{app11031114}
{\sc R.~Arcucci, J.~Zhu, S.~Hu, and Y.-K. Guo}, {\em Deep data assimilation:
  Integrating deep learning with data assimilation}, Applied Sciences, 11
  (2021), \url{https://doi.org/10.3390/app11031114},
  \url{https://www.mdpi.com/2076-3417/11/3/1114}.

\bibitem{bocquet}
{\sc M.~Asch, M.~Bocquet, and M.~Nodet}, {\em Data assimilation}, vol.~11 of
  Fundamentals of Algorithms, Society for Industrial and Applied Mathematics
  (SIAM), Philadelphia, PA, 2016,
  \url{https://doi.org/10.1137/1.9781611974546.pt1},
  \url{https://doi.org/10.1137/1.9781611974546.pt1}.
\newblock Methods, algorithms, and applications.

\bibitem{npg-15-305-2008}
{\sc D.~Auroux and J.~Blum}, {\em A nudging-based data assimilation method: the
  back and forth nudging (bfn) algorithm}, Nonlinear Processes in Geophysics,
  15 (2008), pp.~305--319, \url{https://doi.org/10.5194/npg-15-305-2008},
  \url{https://npg.copernicus.org/articles/15/305/2008/}.

\bibitem{AOT}
{\sc A.~Azouani, E.~Olson, and E.~S. Titi}, {\em Continuous data assimilation
  using general interpolant observables}, J. Nonlinear Sci., 24 (2014),
  pp.~277--304, \url{https://doi.org/10.1007/s00332-013-9189-y},
  \url{https://doi.org/10.1007/s00332-013-9189-y}.

\bibitem{Bessaih2015ContinuousDA}
{\sc H.~Bessaih, E.~Olson, and E.~S. Titi}, {\em Continuous data assimilation
  with stochastically noisy data}, Nonlinearity, 28 (2015), pp.~729--753,
  \url{https://doi.org/10.1088/0951-7715/28/3/729},
  \url{https://doi.org/10.1088/0951-7715/28/3/729}.

\bibitem{biswas}
{\sc A.~Biswas and R.~Price}, {\em Continuous data assimilation for the three
  dimensional navier-stokes equations}, arXiv preprint arXiv:2003.01329,
  (2020).

\bibitem{martinez}
{\sc J.~Blocher, V.~R. Martinez, and E.~Olson}, {\em Data assimilation using
  noisy time-averaged measurements}, Phys. D, 376/377 (2018), pp.~49--59,
  \url{https://doi.org/10.1016/j.physd.2017.12.004},
  \url{https://doi.org/10.1016/j.physd.2017.12.004}.

\bibitem{Brajard2021}
{\sc J.~Brajard, A.~Carrassi, M.~Bocquet, and L.~Bertino}, {\em Combining data
  assimilation and machine learning to emulate a dynamical model from sparse
  and noisy observations: a case study with the {L}orenz 96 model}, J. Comput.
  Sci., 44 (2020), pp.~101171, 11,
  \url{https://doi.org/10.1016/j.jocs.2020.101171},
  \url{https://doi.org/10.1016/j.jocs.2020.101171}.

\bibitem{brajard2021combining}
{\sc J.~Brajard, A.~Carrassi, M.~Bocquet, and L.~Bertino}, {\em Combining data
  assimilation and machine learning to infer unresolved scale parametrization},
  Philosophical Transactions of the Royal Society A, 379 (2021), p.~20200086.

\bibitem{brown2021novel}
{\sc T.~S. Brown, H.~Antil, R.~L{\"o}hner, F.~Togashi, and D.~Verma}, {\em
  Novel dnns for stiff odes with applications to chemically reacting flows},
  arXiv preprint arXiv:2104.01914,  (2021).

\bibitem{box}
{\sc E.~C. Cyr, M.~A. Gulian, R.~G. Patel, M.~Perego, and N.~A. Trask}, {\em
  Robust training and initialization of deep neural networks: An adaptive basis
  viewpoint}, in Mathematical and Scientific Machine Learning, PMLR, 2020,
  pp.~512--536.

\bibitem{MR4298220}
{\sc R.~DeVore, B.~Hanin, and G.~Petrova}, {\em Neural network approximation},
  Acta Numer., 30 (2021), pp.~327--444,
  \url{https://doi.org/10.1017/S0962492921000052},
  \url{https://doi.org/10.1017/S0962492921000052}.

\bibitem{doering}
{\sc C.~R. Doering and J.~Gibbon}, {\em On the shape and dimension of the
  lorenz attractor}, Dynamics and Stability of Systems, 10 (1995),
  pp.~255--268.

\bibitem{Farhat2015ContinuousDA}
{\sc A.~Farhat, M.~S. Jolly, and E.~S. Titi}, {\em Continuous data assimilation
  for the 2{D} {B}\'{e}nard convection through velocity measurements alone},
  Phys. D, 303 (2015), pp.~59--66,
  \url{https://doi.org/10.1016/j.physd.2015.03.011},
  \url{https://doi.org/10.1016/j.physd.2015.03.011}.

\bibitem{Farhat2016OnTC}
{\sc A.~Farhat, E.~Lunasin, and E.~Titi}, {\em On the charney conjecture of
  data assimilation employing temperature measurements alone: The paradigm of
  3d planetary geostrophic model}, Mathematics of Climate and Weather
  Forecasting, 2 (2016).

\bibitem{Farhat2015DataAA}
{\sc A.~Farhat, E.~Lunasin, and E.~S. Titi}, {\em Data assimilation algorithm
  for 3{D} {B}\'{e}nard convection in porous media employing only temperature
  measurements}, J. Math. Anal. Appl., 438 (2016), pp.~492--506,
  \url{https://doi.org/10.1016/j.jmaa.2016.01.072},
  \url{https://doi.org/10.1016/j.jmaa.2016.01.072}.

\bibitem{Farhat2017ADA}
{\sc A.~Farhat, E.~Lunasin, and E.~S. Titi}, {\em A data assimilation
  algorithm: the paradigm of the 3{D} {L}eray-{$\alpha$} model of turbulence},
  in Partial differential equations arising from physics and geometry, vol.~450
  of London Math. Soc. Lecture Note Ser., Cambridge Univ. Press, Cambridge,
  2019, pp.~253--273.

\bibitem{foias2016discrete}
{\sc C.~Foias, C.~F. Mondaini, and E.~S. Titi}, {\em A discrete data
  assimilation scheme for the solutions of the two-dimensional navier--stokes
  equations and their statistics}, SIAM Journal on Applied Dynamical Systems,
  15 (2016), pp.~2109--2142.

\bibitem{Ghosh2020STEERS}
{\sc A.~Ghosh, H.~S. Behl, E.~Dupont, P.~H.~S. Torr, and V.~Namboodiri}, {\em
  Steer : Simple temporal regularization for neural odes}, ArXiv,
  abs/2006.10711 (2020).

\bibitem{paraDNN}
{\sc S.~G\"{u}nther, L.~Ruthotto, J.~B. Schroder, E.~C. Cyr, and N.~R. Gauger},
  {\em Layer-parallel training of deep residual neural networks}, SIAM J. Math.
  Data Sci., 2 (2020), pp.~1--23, \url{https://doi.org/10.1137/19M1247620},
  \url{https://doi.org/10.1137/19M1247620}.

\bibitem{EHaber_LRuthotto_2018a}
{\sc E.~Haber and L.~Ruthotto}, {\em Stable architectures for deep neural
  networks}, Inverse Problems, 34 (2018), pp.~014004, 22,
  \url{https://doi.org/10.1088/1361-6420/aa9a90},
  \url{https://doi.org/10.1088/1361-6420/aa9a90}.

\bibitem{KHe_XZhang_SRen_JSun_2016a}
{\sc K.~He, X.~Zhang, S.~Ren, and J.~Sun}, {\em Deep residual learning for
  image recognition}, in Proceedings of the IEEE Conference on Computer Vision
  and Pattern Recognition, 2016, pp.~770--778.

\bibitem{TheInitializationofNumericalModelsbyaDynamicInitializationTechnique}
{\sc J.~E. Hoke and R.~A. Anthes}, {\em The initialization of numerical models
  by a dynamic-initialization technique}, Monthly Weather Review, 104 (1976),
  pp.~1551 -- 1556,
  \url{https://doi.org/10.1175/1520-0493(1976)104<1551:TIONMB>2.0.CO;2},
  \url{https://journals.ametsoc.org/view/journals/mwre/104/12/1520-0493_1976_104_1551_tionmb_2_0_co_2.xml}.

\bibitem{ji2020stiff}
{\sc W.~Ji, W.~Qiu, Z.~Shi, S.~Pan, and S.~Deng}, {\em Stiff-pinn:
  Physics-informed neural network for stiff chemical kinetics}, arXiv preprint
  arXiv:2011.04520,  (2020).

\bibitem{kellylawstuart}
{\sc D.~T.~B. Kelly, K.~J.~H. Law, and A.~M. Stuart}, {\em Well-posedness and
  accuracy of the ensemble {K}alman filter in discrete and continuous time},
  Nonlinearity, 27 (2014), pp.~2579--2604,
  \url{https://doi.org/10.1088/0951-7715/27/10/2579},
  \url{https://doi.org/10.1088/0951-7715/27/10/2579}.

\bibitem{Law2015DataAA}
{\sc K.~Law, A.~Stuart, and K.~Zygalakis}, {\em Data assimilation}, vol.~62 of
  Texts in Applied Mathematics, Springer, Cham, 2015,
  \url{https://doi.org/10.1007/978-3-319-20325-6},
  \url{https://doi.org/10.1007/978-3-319-20325-6}.
\newblock A mathematical introduction.

\bibitem{364173}
{\sc C.~Le~Provost, J.~Verron, E.~Blayo, J.~Molines, and B.~Barnier}, {\em
  Assimilation of topex/poseidon altimeter data into a circulation model of the
  north atlantic}, in Proceedings of OCEANS'94, vol.~3, 1994,
  pp.~III/63--III/68 vol.3, \url{https://doi.org/10.1109/OCEANS.1994.364173}.

\bibitem{Lorenz96}
{\sc E.~Lorenz}, {\em Predictability: a problem partly solved}, in Seminar on
  Predictability, 4-8 September 1995, vol.~1, Shinfield Park, Reading, 1995,
  ECMWF, ECMWF, pp.~1--18, \url{https://www.ecmwf.int/node/10829}.

\bibitem{lorenz}
{\sc E.~N. Lorenz}, {\em Deterministic nonperiodic flow}, J. Atmospheric Sci.,
  20 (1963), pp.~130--141,
  \url{https://doi.org/10.1175/1520-0469(1963)020<0130:DNF>2.0.CO;2},
  \url{https://doi.org/10.1175/1520-0469(1963)020<0130:DNF>2.0.CO;2}.

\bibitem{majda_harlim_2012}
{\sc A.~J. Majda and J.~Harlim}, {\em Filtering complex turbulent systems},
  Cambridge University Press, Cambridge, 2012,
  \url{https://doi.org/10.1017/CBO9781139061308},
  \url{https://doi.org/10.1017/CBO9781139061308}.

\bibitem{Markowich2015ContinuousDA}
{\sc P.~A. Markowich, E.~S. Titi, and S.~Trabelsi}, {\em Continuous data
  assimilation for the three-dimensional {B}rinkman-{F}orchheimer-extended
  {D}arcy model}, Nonlinearity, 29 (2016), pp.~1292--1328,
  \url{https://doi.org/10.1088/0951-7715/29/4/1292},
  \url{https://doi.org/10.1088/0951-7715/29/4/1292}.

\bibitem{article}
{\sc H.~Nijmeijer}, {\em A dynamical control view on synchronization}, Phys. D,
  154 (2001), pp.~219--228,
  \url{https://doi.org/10.1016/S0167-2789(01)00251-2},
  \url{https://doi.org/10.1016/S0167-2789(01)00251-2}.

\bibitem{pawar2020long}
{\sc S.~Pawar, S.~E. Ahmed, O.~San, A.~Rasheed, and I.~M. Navon}, {\em Long
  short-term memory embedded nudging schemes for nonlinear data assimilation of
  geophysical flows}, Physics of Fluids, 32 (2020), p.~076606.

\bibitem{PINN}
{\sc M.~Raissi, P.~Perdikaris, and G.~E. Karniadakis}, {\em Physics-informed
  neural networks: a deep learning framework for solving forward and inverse
  problems involving nonlinear partial differential equations}, J. Comput.
  Phys., 378 (2019), pp.~686--707,
  \url{https://doi.org/10.1016/j.jcp.2018.10.045},
  \url{https://doi.org/10.1016/j.jcp.2018.10.045}.

\bibitem{reich_cotter_2015}
{\sc S.~Reich and C.~Cotter}, {\em Probabilistic forecasting and {B}ayesian
  data assimilation}, Cambridge University Press, New York, 2015,
  \url{https://doi.org/10.1017/CBO9781107706804},
  \url{https://doi.org/10.1017/CBO9781107706804}.

\bibitem{LRuthotto_EHaber_2019a}
{\sc L.~Ruthotto and E.~Haber}, {\em Deep neural networks motivated by partial
  differential equations}, Journal of Mathematical Imaging and Vision,  (2019),
  \url{https://doi.org/10.1007/s10851-019-00903-1}.

\bibitem{UseofFourDimensionalDataAssimilationinaLimitedAreaMesoscaleModelPartIExperimentswithSynopticScaleData}
{\sc D.~R. Stauffer and N.~L. Seaman}, {\em Use of four-dimensional data
  assimilation in a limited-area mesoscale model. part i: Experiments with
  synoptic-scale data}, Monthly Weather Review, 118 (1990), pp.~1250 -- 1277,
  \url{https://doi.org/10.1175/1520-0493(1990)118<1250:UOFDDA>2.0.CO;2},
  \url{https://journals.ametsoc.org/view/journals/mwre/118/6/1520-0493_1990_118_1250_uofdda_2_0_co_2.xml}.

\end{thebibliography}
\end{document}